\documentclass[11pt,a4paper]{article}
\usepackage{amssymb,amsmath,amsfonts,mathrsfs,bm,enumerate,comment}
\allowdisplaybreaks[1]
\numberwithin{equation}{section}

\newcommand{\fscomment}[1]{\color{blue}}

\usepackage[colorlinks=true, pdfstartview=FitV, linkcolor=blue, citecolor=blue, urlcolor=blue,pagebackref=false]{hyperref}

\usepackage{tikz}
\usepackage{float}
\usepackage[font=footnotesize]{caption}

\usepackage{color} 
\usepackage{times,theorem,latexsym,color,comment}

\newcommand{\BOX}{\ensuremath\Box}

\newtheorem{theorema}{Theorem}

\newtheorem{theorem}{Theorem }[section]

{\theorembodyfont{\rmfamily}}
{\theorembodyfont{\rmfamily}}
{\theorembodyfont{\rmfamily}}
\newtheorem{lemma}[theorem]{Lemma}

{\theorembodyfont{\rmfamily}\newtheorem{remark}[theorem]{Remark}}
{\theorembodyfont{\rmfamily}}

\newcommand{\N}{\mathbb{N}}
\newcommand{\Z}{\mathbb{Z}}
\newcommand{\R}{\mathbb{R}}
\newcommand{\C}{\mathbb{C}}
\newcommand{\rmc}{\mathrm{c}}

\newcommand{\mcO}{\mathcal{O}}
\newcommand{\mcV}{\mathcal{V}}
\newcommand{\dd}{\,{\rm d}}
\newcommand{\opsupp}{\operatorname{supp}}
\newcommand{\opspan}{\operatorname{span}}

\newcommand{\opbog}{\operatorname{Bog}}
\newcommand{\oparg}{\operatorname{arg}}
\newcommand{\ep}{\varepsilon}
\newcommand{\ii}{\mathrm i}

\newcommand{\overbar}[1]{\mkern 1.5mu\overline{\mkern-1.5mu#1\mkern-1.5mu}\mkern 1.5mu}

\def\XXint#1#2#3{{\setbox0=\hbox{$#1{#2#3}{\int}$}
		\vcenter{\hbox{$#2#3$}}\kern-.5\wd0}}

\newenvironment{proof}{{\vskip\baselineskip\noindent\textbf{Proof:}}}%
{\hspace*{.1pt}\hspace*{\fill}\BOX\vskip\baselineskip}

\newenvironment{proofx}[1]%
{\vskip\baselineskip\noindent\textbf{Proof of {#1}:}}%
{\hspace*{.1pt}\hspace*{\fill}\BOX\vskip\baselineskip}
{\vskip\baselineskip\noindent\textbf{Proof of Theorem \protect\ref{#1}:}}%
{\hspace*{.1pt}\hspace*{\fill}\BOX\vskip\baselineskip}
{\vskip\baselineskip\noindent\textbf{Proof of Theorems \protect\ref{#1} --
		\protect\ref{#2}:}}%
{\hspace*{.1pt}\hspace*{\fill}\BOX\vskip\baselineskip}

\begin{document}

\title{A Runge-type approximation theorem \\
for the 3D unsteady Stokes system}

\author{Mitsuo Higaki and Franck Sueur}
\date{}

\maketitle

\noindent {\bf Abstract.}\ 
We investigate Runge-type approximation theorems for solutions to the 3D unsteady Stokes system. More precisely, we establish that on any compact set with connected complement, local smooth solutions to the 3D unsteady Stokes system can be approximated with an arbitrary small positive error in $L^\infty$ norm by a global solution of the 3D unsteady Stokes system, where the velocity grows at most exponentially at spatial infinity and the pressure grows polynomially. Additionally, by considering a parasitic solution to the Stokes system, we establish that some growths at infinity are indeed necessary.

\tableofcontents

    \section{Introduction and statement of the main results}
    \label{sec.intro}

    \subsection{Runge-type approximation theorems and the 3D unsteady Stokes system}

A classical theorem by Runge \cite{Runge} states, in particular, that any function which is holomorphic in a neighborhood of a compact set $K$ of the complex plane $\C$, such that the complementary set $\C \setminus K$ is connected, can be uniformly approximated by some polynomials. Actually the full version of Runge's theorem allows for any domain whose complement is disconnected up to extending the space of approximations to the space of rational functions with at least one pole in each connected component of the complementary.

This result has been the subject of many extensions. First the Mergelyan theorem \cite{Mergelyan} in 1952 concerns functions which are only continuous on $K$ and holomorphic in the interior of $K$. The result was also extended to the Laplace equations in any dimension by Walsh \cite{Walsh} in 1929. Then Lax \cite{Lax} and Malgrange \cite{Malgrange} proved that for some classes of elliptic linear differential equations in the Euclidean space $\R^d$, solutions in a bounded set $\Omega$ can be uniformly approximated by solutions in $\R^d$ provided that $\Omega$ has a connected complement in $\R^d$. The classes of admissible elliptic operators, whose main example is the set of operators with analytic coefficients, were subsequently extended by Browder \cite{Browder}, including approximation in certain H\"older spaces. Later on some results were obtained for non-elliptic equations; see for example \cite{Jon1975,EGFPS2019,EncPer-Sal2021}.

In this paper, we investigate Runge-type properties of the 3D unsteady Stokes system
\begin{equation}\tag{S}\label{intro.eq.S}
\left\{
\begin{array}{ll}
\partial_{t} v - \Delta v + \nabla q 
= 0, \\
\nabla\cdot v = 0.
\end{array}\right.
\end{equation}
That is, the possibility to approximate local solutions of \eqref{intro.eq.S} by some global solutions. By a local solution, we mean a smooth solution $(v,q)$ of \eqref{intro.eq.S} on some compact set $K\subset\R^{4}_{+}$, where $\R^{4}_{+}$ denotes the full space for positive times $\R^{3} \times (0,\infty)$, in the sense that $(v,q)$ satisfies \eqref{intro.eq.S} for some open neighborhood $\mcO$ of $K$ in $\R^{4}_{+}$ and is smooth in $\mcO$. On the other hand, by a global solution, we mean a smooth solution $(u,p)$ of the Stokes system in $\R^{4}_{+}$
\begin{equation}\label{app}
\left\{
\begin{array}{ll}
\partial_{t} u - \Delta u + \nabla p
= 0&\mbox{in}\ \R^{4}_{+},\\[2pt]
\nabla\cdot u
= 0&\mbox{in}\ \R^{4}_{+}.
\end{array}\right.
\end{equation}
In the sequel an important topological condition on $K$ is that the time-sections $K(t) = K \cap (\R^{3}\times \{ t \})$, for all $t\in\R$, have complementary sets $K^{\rmc}(t)$ which are connected.

Our work is carried out in the class of functions without decay in space. The analysis of the unsteady Stokes system in such classes is known to present particular difficulties, with which we face in this paper can be summarized as: (I) Unboundedness of the Riesz operators in the $L^{\infty}$ space; (I\hspace{-1.2pt}I) Existence of certain exact solutions, called parasitic solutions.

Technical challenges may arise from Point (I). Recall that the Helmholtz projection (see \cite{Soh2013book}) in $\R^{3}$ involves the Riesz operators. In particular, when proving the global approximation theorems using the $L^{\infty}$ norm, it is not clear that simply revisiting the approach in \cite{EGFPS2019} for the parabolic equations will be successful. Somewhat surprisingly, it is figured out in Sections \ref{Ext}--\ref{App-disc} that at least in the ``sweeping of poles and discretization" argument, which is the classical and first step towards Runge-type theorems, the issues related to Point (I) do not have any effect. This is thanks to off-diagonal nature of the argument. Consequently, we obtain a Runge-type approximation theorem as described in Theorem \ref{thma} in the next section, which is in the same vein as \cite[Theorems 1.1 and 1.2]{EGFPS2019}.

Difficulties arising from Point (I\hspace{-1.2pt}I) are more inherent in this paper. In fact, we point out that, unlike \cite[Theorem 1.2]{EGFPS2019}, the approximations in Theorem \ref{thma} exhibit growth in space at spatial infinity rather than decay. This relates to the lack of uniqueness of solutions for the Cauchy problem of the Stokes system within the class of functions without decay at infinity. To see the breakdown of uniqueness, we refer to the following exact solutions
\begin{equation}\label{para}
v(x,t) = c(t) \nabla h(x)
\qquad \text{and} \qquad
q(x,t) = -c'(t) h(x),
\end{equation}
where $c(t)$ is a function on $(0,\infty)$ and $h(x)$ is a harmonic function on $\R^{3}$. The solutions \eqref{para} are called parasitic solutions or Serrin's examples \cite{Ser1962}. Note that parasitic solutions play a crucial role in the local regularity theory of the (Navier-)Stokes system \cite{Ser2015book} since they indicate that the regularity of solutions in time is determined by the pressure, and hence, there is no smoothing in time in general. The presence of parasitic solutions stands in sharp contrast to the situation for the heat equation or, in particular, for \cite[Theorem 1.2]{EGFPS2019}, where one can appeal to the uniqueness of solutions with exponential growth in space; see \cite[Section 7.1 (b)]{Joh1991book}. Furthermore, parasitic solutions provide counterexamples to \cite[Theorem 1.2]{EGFPS2019} for the case of the Stokes system, as described in Theorem \ref{thmb} in the next section. It states that there are parasitic solutions which cannot be approximated by bounded global solutions of the Stokes system. Therefore, some growths at infinity such as \eqref{growth.thma} in Theorem \ref{thma} are indeed necessary for the Runge-type theorems to hold.

    \subsection{A positive result with some growth at infinity}

Our main result in this paper is the following.

\begin{theorema}\label{thma}
Let $K\subset \R^{4}_{+}$ be a compact set such that $K^{\rmc}(t)$ is connected for all $t\in\R$. Given a smooth solution $(v,q)$ of the Stokes system \eqref{intro.eq.S} in $K$ and $\ep>0$, there is a smooth solution $(u,p)$ of \eqref{app} and positive real numbers $C,c,\rho$ such that for all $(x,t)$ in $\R^{4}_{+}$
\begin{equation}\label{growth.thma}
|u(x,t)| \le C e^{c|x|}
\qquad \text{and} \qquad
|p(x,t)| \le C (1+|x|)^{\rho} ,
\end{equation}
and
\begin{equation}\label{approx.thma}
 \|v - u\|_{C(K)}  < \ep.   
\end{equation}
\end{theorema}

The proof of Theorem \ref{thma} is given in Section \ref{sec-proofA}.

\begin{remark}\label{rem.thma}
\begin{enumerate}[(i)]
\item\label{item1.rem.thma}
The approximation \eqref{approx.thma} can be improved in space to, for fixed $l\in\Z_{\ge0}$,
\begin{align*}
\|v - u\|_{C^{l}_{x}C_{t}(K)}
< \ep.
\end{align*}
Indeed, one can apply to $v-u$ the regularity theory of the steady Stokes system \cite{Ser2015book} at each fixed $t$ by regarding $-(\partial_{t} v - \partial_{t} u)$ as a given forcing. Improvement in time cannot be expected due to the loss of smoothing implied by parasitic solutions \eqref{para}.

\item\label{item3.rem.thma}
The statement of Theorem \ref{thma} in $\R^{4}_{+}$ for positive times is presented to underline the difference with \cite[Theorem 1.2]{EGFPS2019} for the heat equations, although the global approximations $(u,p)$ can be actually defined in the whole space-time $\R^{4}$ as seen in the proof of Theorem \ref{thma}. In \cite[Theorem 1.2]{EGFPS2019}, the approximations are expressed as $u=e^{t\Delta} u_{0}$ using the heat semigroup $e^{t\Delta}$ and a suitably chosen smooth, compactly supported initial data $u_{0}$. We emphasize that the difference stems from the uniqueness of solutions to the Cauchy problems in the class of functions with (exponential) growth in space. More details on this point can be found in Section \ref{Disc}.
\end{enumerate}
\end{remark}

    \subsection{A negative result without any growth at infinity}

A complementary result to Theorem \ref{thma}, which states that one cannot expect a similar result without some growth of the approximation at infinity, is the following.

\begin{theorema}\label{thmb}
Let $K\subset \R^{4}_{+}$ be a flat compact space-time cylinder. There are a smooth parasitic solution $(v,q)$ of the form \eqref{para} and $\ep>0$ such that the following holds: if $(u,p)$ is a smooth solution of \eqref{app} which is bounded, that is which satisfies, for all $(x,t)$ in $\R^{4}_{+}$,
\begin{equation}\label{bdd.thmb}
|u(x,t)| +|p(x,t)|
\le C
\end{equation}
for some constant $C$, then 
\begin{align*}
\|v - u\|_{C(K)}
> \ep .
\end{align*}
\end{theorema}

The proof of Theorem \ref{thmb} is given in Section \ref{sec-proofB}.

\begin{remark}
\begin{enumerate}[(i)]
\item
For the heat equation, \cite[Theorem 1.2]{EGFPS2019} shows that any local solution $v$ can be approximated by a global solution $u=e^{t\Delta} u_{0}$ to the Cauchy problem with a suitably chosen smooth, compactly supported initial data $u_{0}$. Theorem \ref{thmb} asserts that the analogue of \cite[Theorem 1.2]{EGFPS2019} is not true for the Stokes system.

\item
Theorem \ref{thmb} prevents us from investing the cost problem by a bounded smooth global approximation following what is done in \cite{RulSal2019,EncPer-Sal2021} for some other systems.

\item
Noticeably the situation is quite different from Theorem \ref{thmb} in the case of fractional Fourier operators, in particular regarding the reachable states; see \cite{DSV2019}.
\end{enumerate}
\end{remark}

The rest of this paper is organized as follows. In Section \ref{sec.prelim}, we summarize the facts about vector calculus in the spherical coordinates and about the modified Bessel functions utilized in this paper. Section \ref{sec-proofA} is devoted to the proof of Theorem \ref{thma} and Section \ref{sec-proofB} to Theorem \ref{thmb}. An in-depth discussion of Theorem \ref{thma} is provided in Section \ref{Disc}.

    \section{Preliminaries}
    \label{sec.prelim}

    \subsection{Notation}

Throughout the paper, we adopt the following notation.
\begin{itemize}
\item
We denote the component of a vector $x\in\R^{3}$ by $x=(x^{1}\, x^{2}\, x^{3})^{\top}$.

\item
We write $A\lesssim B$ if there is some constant $C$, called the implicit constant, such that $A\le CB$. We also write $A\approx B$ if there is some implicit constant $C\ge1$ such that $C^{-1}A\le B\le CA$. The dependence of an implicit constant on other parameters $a,b,c,\ldots$ will be indicated as $A\lesssim_{\,a,b,c,\ldots} B$ or $A\approx_{\,a,b,c,\ldots} B$.

\item
We let $B_{R}$ denote the ball in $\R^{3}$ with radius $R>0$ centered at the origin.

\item
For $z\in\C\setminus (-\infty,0]$, we take the square root $\sqrt{z}$ so that $\Re \sqrt{z}>0$.

\item
When no confusion may occur, we use the same symbol to denote the quantities concerning scalar- or vector-valued mappings. For example, $\langle \cdot, \cdot\rangle_{S}$ denotes the inner product on $L^2(S)$ or $L^2(S)^2$, each of which is defined in the next section. 
\end{itemize}

    \subsection{Spherical coordinates}

Let $(r,\theta,\phi)$ denote the spherical coordinates on $\R^{3}$
\begin{align}\label{def.sph.coord.}
\begin{split}
\left(
\begin{matrix}
x^{1}\\
x^{2}\\
x^{3}
\end{matrix}
\right)
=
\left(
\begin{matrix}
r\sin \theta \cos \phi\\
r\sin \theta \sin \phi\\
r\cos \theta
\end{matrix}
\right),
\quad
(r,\theta,\phi)
\in
(0,\infty)
\times
[0,\pi]
\times
[0,2\pi)
\end{split}
\end{align}
and $\bm{\hat r},\bm{\hat \theta},\bm{\hat \phi}$ the unit vectors
\begin{align}
\bm{\hat r}
=
\left(
\begin{matrix}
\sin \theta \cos \phi\\
\sin \theta \sin \phi\\
\cos \theta
\end{matrix}
\right),
\quad
\bm{\hat \theta}
=
\left(
\begin{matrix}
\cos \theta \cos \phi\\
\cos \theta \sin \phi\\
-\sin \theta
\end{matrix}
\right),
\quad
\bm{\hat \phi}
=
\left(
\begin{matrix}
-\sin \phi\\
\cos \phi\\
0
\end{matrix}
\right).
\end{align}
For a function $f$ on $\R^{3}$, we slightly abuse the notation to denote
\[
f(x)
=f(r,\theta,\phi)
=f(r\sin\theta \cos\phi, r\sin\theta \sin\phi, r\cos\theta).
\]

Let $S_{r}$ denote the sphere with radius $r$ in $\R^{3}$:
\[
S_{r}=\{x\in\R^{3}~|~|x|=r\},
\quad r\in(0,\infty)
\]
and $S=S_{1}$ the unit sphere. For a function $f$ on $S$, we denote
\[
f(x)
=f(\theta,\phi)
=f(\sin\theta \cos\phi, \sin\theta \sin\phi, \cos\theta).
\]
The integral over $S$ is defined as
\[
\int_{S} f
=
\int_{0}^{2\pi}
\bigg(
\int_{0}^{\pi} f(\theta,\phi) \sin\theta \dd \theta
\bigg)
\dd \phi.
\]
For $f,g\in L^{2}(S)$, we write
\[
\langle f,g\rangle_{S} = \int_{S} f \overbar{g},
\qquad
\|f\|_{L^2(S)} = \langle f,f\rangle_{S}^\frac{1}{2}.
\]
Following our notational convention, we define in the same manner the space $L^{2}(S)^{3}$ for vector fields on $S$ as well as the inner product $\langle \cdot, \cdot\rangle_{S}$ on $L^{2}(S)^{3}$ and the norm $\|\cdot\|_{L^2(S)}$.

The Laplace-Beltrami operator on $S$ is defined by its action on a scalar field $f$ by
\begin{equation} \label{DS}
   \Delta_S f
 =   \frac{1}{\sin\theta} \frac{\partial}{\partial \theta}\left(\sin\theta \frac{\partial f}{\partial \theta}\right)
+ \frac{1}{\sin^2\theta} \frac{\partial^2 f}{\partial \phi^2} .
\end{equation}
Then the Laplacian of a scalar field $f$ in spherical coordinates is written as
\begin{equation} \label{DpasS}
   \Delta f
 =   \frac{1}{r^2} \frac{\partial}{\partial r}\left(r^2 \frac{\partial f}{\partial r}\right)
 +  \frac{1}{r^2} \Delta_S f .
\end{equation}
We also use the following: for a scalar field $f$, 
\begin{equation}
\nabla f
= 
\frac{\partial f}{\partial r}
\bm{\hat r}
+ \frac{1}{r} \frac{\partial f}{\partial \theta}
\bm{\hat \theta}
+ \frac{1}{r\sin\theta} \frac{\partial f}{\partial \phi}
\bm{\hat \phi}.
\end{equation}
For a vector field $w$, we will use the operators
\begin{align}
\begin{split}
\Delta w 
&=
\Big(
\Delta w^{r}
- \frac{2 w^{r}}{r^2}
- \frac{2}{r^2\sin\theta}
\frac{\partial}{\partial\theta} (w^{\theta} \sin\theta)
- \frac{2}{r^2\sin\theta}
\frac{\partial w^{\phi}}{\partial \phi}
\Big)
\bm{\hat r}\\
&\quad
+
\Big(\Delta w^{\theta}
- \frac{w^{\theta}}{r^2\sin^2\theta}
+ \frac{2}{r^2} \frac{\partial w^{r}}{\partial\theta}
- \frac{2\cos\theta}{r^2\sin^2\theta}
\frac{\partial w^{\phi}}{\partial \phi}
\Big)
\bm{\hat \theta}\\
&\quad
+
\Big(\Delta w^{\phi}
- \frac{w^{\phi}}{r^2\sin^2\theta}
+ \frac{2}{r^2\sin\theta} \frac{\partial w^{r}}{\partial \phi}
+ \frac{2\cos\theta}{r^2\sin^2\theta}
\frac{\partial w^{\theta}}{\partial \phi}
\Big)
\bm{\hat \phi}
\end{split}
\end{align}
and
\begin{align}
\begin{split}
\nabla\cdot w
&=
\frac{1}{r^{2}}
\frac{\partial}{\partial r}(w^{r} r^{2})
+ \frac{1}{r\sin\theta}
\frac{\partial}{\partial \theta}(w^{\theta} \sin\theta)
+ \frac{1}{r\sin\theta}
\frac{\partial w^{\phi}}{\partial \phi},\\
\nabla\times w
&=
\frac{1}{r\sin\theta}
\Big\{
\frac{\partial}{\partial \theta}(w^{\phi} \sin\theta)
- \frac{\partial w^{\theta}}{\partial \phi}
\Big\}
\bm{\hat r}\\
&\quad
+ \frac{1}{r}
\Big\{
\frac{1}{\sin\theta}
\frac{\partial w^{r}}{\partial \phi}
- \frac{\partial}{\partial r}(r w^{\phi})
\Big\}
\bm{\hat \theta}\\
&\quad
+ \frac{1}{r}
\Big\{
\frac{\partial}{\partial r}(r w^{\theta})
- \frac{\partial w^{r}}{\partial \theta}
\Big\}
\bm{\hat \phi}.
\end{split}
\end{align}

    \subsection{Scalar/Vector spherical harmonics}
    \label{sec.SH}

Let $Y_{lm}=Y_{lm}(\theta,\phi)$
denote the (scalar) spherical harmonics on $S$ defined by
\begin{equation}
Y_{lm}(\theta,\phi)
=\sqrt{\frac{2l+1}{4\pi}\frac{(l-m)!}{(l+m)!}}P_{lm}(\cos\theta) e^{\ii m\phi},
\quad
l\in \Z_{\ge0},
\mkern9mu
m=0,\pm 1,\ldots,\pm l
\label{fi}.
\end{equation}
Here $P_{lm}=P_{lm}(y)$ is the function defined by
\begin{equation*}
P_{lm}(x)
=
\left\{
\begin{array}{ll}
(-1)^{m}
(1-y^2)^{\frac{m}{2}}
\dfrac{\dd^{m}}{\dd y^{m}}
P_{l0}(y)
&\mbox{if}\ m\ge0,\\[10pt]
(-1)^{|m|}
\dfrac{(l-|m|)!}{(l+|m|)!}
P_{l|m|}(y)
&\mbox{if}\ m<0
\end{array}\right.
\end{equation*}
associated with the Legendre polynomial
\begin{equation}
P_{l0}(y)
=
\frac{1}{2^{l}l!}
\dfrac{\dd^{l}}{\dd y^{l}}
(y^{2}-1)^{l}.
\end{equation}
The spherical harmonics satisfy the orthogonality relations
\begin{equation} \label{sortho}
\langle Y_{lm}, Y_{l'm'} \rangle_{S}
=
\delta_{l, l'}\delta_{m, m'},
\end{equation}
where $\delta_{j,k}$ denotes the Kronecker delta.

\begin{lemma}\label{lem.SH.exp}
The following hold.
\begin{enumerate}[(1)]
\item\label{item1.lem.SH.exp}
For $f\in L^{2}(S)$, we have
\begin{align}\label{exp.S.SH}
f=
\sum_{l=0}^{\infty} \sum_{m=-l}^{l} 
\langle f, Y_{lm} \rangle_{S}
Y_{lm}.
\end{align}
The series on the right converges in the norm $\|\cdot\|_{L^{2}(S)}$.
\item\label{item2.lem.SH.exp}
For $l\in\Z_{\ge0}$, we have
\begin{align}\label{def.mu}
\Delta_{S} Y_{lm}  = - \mu_{l} Y_{lm},
\qquad 
\mu_{l} := l(l+1).
\end{align}

\item\label{item3.lem.SH.exp}
For $f\in C^\infty(S)$, we have
\begin{align}
|\langle f, Y_{lm} \rangle_{S}|
\lesssim_{\,f,k}
(1+l)^{-2k},
\quad
l\in\Z_{\ge0}.
\end{align}

\item\label{item4.lem.SH.exp}
For $f\in C^\infty(S)$, the series in \eqref{exp.S.SH} converges absolutely and uniformly on $S$.
\end{enumerate}
\end{lemma}

\begin{proof}
(\ref{item1.lem.SH.exp})
This is well known and documented, for example, in \cite[Section 5.6]{VMK1988book}.

(\ref{item2.lem.SH.exp})
This follows from the direct calculation of the left of $\Delta(r^l Y_{lm})=0$ using \eqref{DpasS}.

(\ref{item3.lem.SH.exp})
By integration by parts,
\[
-\mu_{l} \langle f, Y_{lm} \rangle_{S}
= \langle f, \Delta_{S} Y_{lm} \rangle_{S}
= \langle \Delta_{S} f, Y_{lm} \rangle_{S}.
\]
In the same manner, we have
\[
(-\mu_{l})^{k} \langle f, Y_{lm} \rangle_{S}
= \langle \Delta_{S}^{k} f, Y_{lm} \rangle_{S},
\quad k\in \Z_{\ge0}.
\]
Thus the assertion follows from
\[
|\langle f, Y_{lm} \rangle_{S}|
\le 
(\mu_{l}+1)^{-k}
\|\Delta_{S}^{k} f\|_{L^{2}(S)}
\|Y_{lm}\|_{L^{2}(S)}
=
(\mu_{l}+1)^{-k}
\|\Delta_{S}^{k} f\|_{L^{2}(S)}.
\]

(\ref{item4.lem.SH.exp})
By the Sobolev embedding theorem and by elliptic regularity of $\Delta_{S}$, we have
\begin{align}
\|Y_{lm}\|_{L^{\infty}(S)}
\lesssim
\|Y_{lm}\|_{W^{2,2}(S)}
\lesssim
\|\Delta_{S} Y_{lm}\|_{L^{2}(S)} + 1
\lesssim
\mu_{l}+1.
\end{align}
which combined with (\ref{item3.lem.SH.exp}) implies the assertion. The proof is complete.
\end{proof}

Next we introduce vector spherical harmonics with reference to \cite{BEG1985, VMK1988book}. In this paper, we will mainly adopt the notation used in \cite{BEG1985}. Let
\[
\bm{Y}_{lm}=\bm{Y}_{lm}(\theta,\phi),
\qquad
\bm{\Psi}_{lm}=\bm{\Psi}_{lm}(\theta,\phi),
\qquad
\bm{\Phi}_{lm}=\bm{\Phi}_{lm}(\theta,\phi)
\]
denote the vector spherical harmonics defined by
\begin{align}\label{def.VSH}
\begin{split}
\bm Y_{lm} = Y_{lm} \bm{\hat r},
\qquad
\bm \Psi_{lm} = r\nabla Y_{lm},
\qquad
\bm \Phi_{lm} = \bm{r} \times\nabla Y_{lm}.
\end{split}
\end{align}

More explicitly, we have
\begin{align}\label{Psi.Phi}
\begin{split}
\bm \Psi_{lm}(\theta,\phi) 
&=
(\partial_{\theta}Y_{lm}(\theta,\phi)) \bm{\hat \theta}
+ \frac{\partial_{\phi}Y_{lm}(\theta,\phi)}{\sin\theta} \bm{\hat \phi},
\quad l\ge1,\\
\bm \Phi_{lm}(\theta,\phi)
&
=
(\partial_{\theta}Y_{lm}(\theta,\phi)) \bm{\hat \phi}
- \frac{\partial_{\phi}Y_{lm}(\theta,\phi)}{\sin\theta} \bm{\hat \theta},
\quad l\ge1
\end{split}
\end{align}
Notice that, since $Y_{00}=\sqrt{1/4\pi}$,
\begin{align}\label{Psi.Phi.00}
\bm \Psi_{00} = \bm \Phi_{00} = 0.
\end{align}
The vector spherical harmonics satisfy the following relations:
\begin{align*}
\langle \bm{Y}_{lm}, \bm{Y}_{l'm'} \rangle_{S} 
&=\delta_{l,l'}\delta_{m,m'},\\
\langle \bm{\Psi}_{lm}, \bm{\Psi}_{l'm'} \rangle_{S} 
&=\langle \bm{\Phi}_{lm}, \bm{\Phi}_{l'm'} \rangle_{S}
= l(l+1)\delta_{l,l'}\delta_{m,m'},\\
\langle \bm{Y}_{lm}, \bm{\Psi}_{l'm'} \rangle_{S} 
&=\langle \bm{Y}_{lm}, \bm{\Phi}_{l'm'} \rangle_{S} 
=\langle \bm{\Psi}_{lm}, \bm{\Phi}_{l'm'} \rangle_{S}
=0.
\end{align*}

Let $f$ be a function of the variable $r$. A direct calculation leads to
\begin{align}\label{nabla.VSH}
\nabla (f Y_{lm})
= \frac{\dd f}{\dd r} \bm{Y}_{lm}
+ \frac{f}{r} \bm{\Psi}_{lm}.
\end{align}
Moreover, recalling $\mu_{l}=l(l+1)$ in \eqref{def.mu}, we have
\begin{align}\label{div.VSH}
\begin{split}
\nabla\cdot (f \bm{Y}_{lm})
&=
\Big(
\frac{\dd f}{\dd r} + \frac{2f}{r}
\Big) Y_{lm},\\
\nabla\cdot (f \bm{\Psi}_{lm})
&=
-\frac{\mu_{l} f}{r} Y_{lm},\\
\nabla\cdot (f \bm{\Phi}_{lm})
&=0
\end{split}
\end{align}
and
\begin{align}\label{rot.VSH}
\begin{split}
\nabla\times (f \bm{Y}_{lm})
&=
-\frac{f}{r} \bm \Phi_{lm},\\
\nabla\times (f \bm{\Psi}_{lm})
&=
\Big(
\frac{\dd f}{\dd r} + \frac{f}{r}
\Big) \bm \Phi_{lm},\\
\nabla\times (f \bm{\Phi}_{lm})
&=
-\frac{\mu_{l} f}{r} \bm{Y}_{lm}
- \Big(
\frac{\dd f}{\dd r} + \frac{f}{r}
\Big) \bm \Psi_{lm}.
\end{split}
\end{align}

\begin{lemma}\label{lem.VSH.exp}
The following hold.
\begin{enumerate}[(1)]
\item\label{item1.lem.VSH.exp}
For $w\in L^{2}(S)^{3}$, we have
\begin{equation}\label{exp.S.VSH}
w
= \sum_{l=0}^{\infty} \sum_{m=-l}^{l}
\big(
w^{r}_{lm} \bm{Y}_{lm}
+ w^{(1)}_{lm} \bm{\Psi}_{lm}
+ w^{(2)}_{lm} \bm{\Phi}_{lm}
\big)
\end{equation}
with the coefficients $w^{r}_{lm},w^{(1)}_{lm},w^{(2)}_{lm}$ given by
\begin{align}\label{coef.exp.S.VSH}
\begin{split}
w^{r}_{lm}
&=
\langle w, \bm{Y}_{lm}\rangle_{S},\\
w^{(1)}_{lm}
&=
\frac{1}{l(l+1)}
\langle w, \bm{\Psi}_{lm}\rangle_{S},
\quad l\ge1\\
w^{(2)}_{lm}
&=
\frac{1}{l(l+1)}
\langle w, \bm{\Phi}_{lm}\rangle_{S},
\quad l\ge1.
\end{split}
\end{align}
The series on the right converges in the norm $\|\cdot\|_{L^{2}(S)}$.

\item\label{item2.lem.VSH.exp}
For $w\in C^\infty(S)^{3}$, the series in \eqref{exp.S.VSH} converges absolutely and uniformly on $S$.
\end{enumerate}
\end{lemma}

\begin{proof}
(\ref{item1.lem.VSH.exp})
This is well known and is documented, for example, in \cite[Section 7.3]{VMK1988book}. We remark that the vector spherical harmonics $\bm{Y}_{lm}, \bm{\Psi}_{lm}, \bm{\Phi}_{lm}$ correspond to the ones 
$\bm{Y}_{lm}^{(1)}, \bm{Y}_{lm}^{(0)}, \bm{Y}_{lm}^{(-1)}$ in \cite[Section 7.3.1]{VMK1988book} 
 by the linear transformation
\[
\begin{bmatrix}
\bm{Y}_{lm}\\[3pt]
\bm{\Psi}_{lm}\\[3pt]
\bm{\Phi}_{lm}\\
\end{bmatrix}
=
\begin{bmatrix}
0 & 0 & 1\\[3pt]
\sqrt{l(l+1)} & 0 & 0\\[3pt]
0 & \ii \sqrt{l(l+1)} & 0\\
\end{bmatrix}
\begin{bmatrix}
\bm{Y}_{lm}^{(1)}\\[3pt]
\bm{Y}_{lm}^{(0)}\\[3pt]
\bm{Y}_{lm}^{(-1)}\\
\end{bmatrix}.
\]
For a detailed proof, otherwise, the reader may consult \cite[Section 5.4]{FreGut2013book}.

(\ref{item2.lem.VSH.exp})
This can be proved in a similar manner to Lemma \ref{lem.SH.exp} (\ref{item4.lem.SH.exp}) using the definition of $\bm Y_{lm}, \bm \Psi_{lm}, \bm \Phi_{lm}$ in \eqref{def.VSH} and Lemma \ref{lem.SH.exp} (\ref{item3.lem.SH.exp}). The proof is complete.
\end{proof}

For a smooth vector field $w$ on an open ball $B_{\rho}$ of radius $\rho$, we write
\begin{equation}\label{exp.VSH}
w(r,\theta,\phi)
= \sum_{l=0}^{\infty} \sum_{m=-l}^{l}
\big(
w^{r}_{lm}(r) \bm{Y}_{lm}
+ w^{(1)}_{lm}(r) \bm{\Psi}_{lm}
+ w^{(2)}_{lm}(r) \bm{\Phi}_{lm}
\big)
\end{equation}
with the coefficients $w^{r}_{lm},w^{(1)}_{lm},w^{(2)}_{lm}$ defined in \eqref{coef.exp.S.VSH}. If $w$ is smooth on $B_{R}$ for some $R>0$, the series in \eqref{exp.VSH} converges absolutely and uniformly on $B_{\rho}$ with any $0<\rho<R$.

    \subsection{Modified Bessel function}
    \label{MBF}

We collect here some facts about the modified Bessel functions. References are \cite{Wat1944book,AAR1999book}. The modified Bessel function of the first kind $I_\nu(z)$ of order $\nu$ is defined by 
\begin{align}\label{def.I}
\begin{split}
I_\nu(z) 
= 
\Big(\frac{z}{2}\Big)^\nu 
\sum_{k=0}^{\infty} 
\frac{1}{k!\Gamma(\nu+k+1)} \Big(\frac{z}{2}\Big)^{2k}, 
\quad 
z\in \C\setminus (-\infty,0], 
\end{split}
\end{align}
where $\Gamma(z)$ denotes the Gamma function, the second kind $K_\nu(z)$ of order $\nu\notin\Z$ is by 
\begin{align}\label{def.K}
K_\nu(z) 
= 
\frac{\pi}{2}
\frac{I_{-\nu}(z) - I_\nu(z)}{\sin(\nu \pi)}, 
\quad 
z\in \C\setminus (-\infty,0],
\end{align}
and $K_n(z)$ of order $n\in\Z$ is by the limit of $K_\nu(z)$ in \eqref{def.K} as $\nu\to n$. Moreover, $K_\nu(z)$ and $I_\nu(z)$ are linearly independent solutions of 
\[
-\frac{\dd^2 \omega}{\dd z^2} 
- \frac{1}{z} \frac{\dd \omega}{\dd z} 
+ \Big(1+\frac{\nu^2}{z^2}\Big) \omega = 0
\]
and the Wronskian is
\begin{align}\label{def.W}
\det 
\left(
\begin{array}{cc}
K_\nu(z) 
& I_\nu(z) \\ [5pt]
\displaystyle{\frac{\dd K_\nu}{\dd z}(z)} 
& \displaystyle{\frac{\dd I_\nu}{\dd z}(z)}
\end{array}
\right)
=
\frac1z. 
\end{align}
It is known that $I_\nu(z)$ and $K_\nu(z)$ satisfy
\begin{align}\label{eq.rel}
\begin{split}
\frac{\dd I_\nu}{\dd z}(z)
&
=\frac{\nu}{z} I_\nu(z) + I_{\nu+1}(z)
=-\frac{\nu}{z} I_\nu(z) + I_{\nu-1}(z), \\
\frac{\dd K_\nu}{\dd z}(z) 
&
=\frac{\nu}{z} K_\nu(z) - K_{\nu+1}(z)
=-\frac{\nu}{z} K_\nu(z) - K_{\nu-1}(z). 
\end{split}
\end{align}

The following lemma collects the properties of $I_\nu(z),K_\nu(z)$. Each of them easily follows from the formulas in \cite{Wat1944book,AAR1999book}. Let $\Sigma_{\phi} = \{z\in\C\setminus\{0\}~|~|\oparg z|<\phi\}$.

\begin{lemma}\label{lem.MBF}
Let $\nu>0$, $\delta\in(0,\pi/2)$ and $M>0$. Then the following hold.
\begin{enumerate}[(1)]
\item\label{item1.lem1.MBF}
For $z \in\Sigma_{\pi/2} \cap \{|z|< M\}$, we have
\[
|K_{\nu}(z)|
\approx_{\,\nu,\delta,M}
|z|^{-\nu},
\qquad
|I_{\nu}(z)|
\approx_{\,\nu,\delta,M}
|z|^{\nu}.
\]

\item\label{item2.lem1.MBF}
For $z\in\Sigma_{\pi/2-\delta} \cap \{|z|\ge M\}$, we have
\[
|K_{\nu}(z)|
\approx_{\,\nu,\delta,M}
|z|^{-1/2} e^{-\Re z},
\qquad
|I_{\nu}(z)|
\approx_{\,\nu,\delta,M}
|z|^{-1/2} e^{\Re z}.
\]
\end{enumerate}
\end{lemma}

    \section{Proof of Theorem \ref{thma}}
    \label{sec-proofA}

In this section, we prove Theorem \ref{thma}. Let $K\subset \R^{4}_{+}$ be a compact set such that $K^{\rmc}(t)$ is connected for all $t\in\R$. Let $(v,q)$ be a solution of the Stokes system \eqref{intro.eq.S} on some open neighborhood $\mcO$ of $K$ and assume that $(v,q)$ is smooth in $\mcO$. Let $\ep>0$ be given.

Without loss of generality, we may assume that
\begin{align}\label{assum1.O}
K
\subset V
\Subset \mcO
\subset B_{R}\times(0,T)
\end{align}
for some open set $V$ and $R,T>0$. Furthermore, we may assume that $\mcO$ and $V$ have the same number of finite connected components, that is, in disjoint unions,
\begin{align}\label{assum2.O}
\mcO = \mcO_{1}\sqcup \cdots \sqcup \mcO_{N},
\qquad
V = V_{1}\sqcup \cdots \sqcup V_{N}
\end{align}
and that the following hold: for all $i\in\{1,\ldots,N\}$,
\begin{equation}\label{assum3.O}
\left\{
\begin{array}{ll}
V_{i}\subset \mcO_{i},\\[2pt]
\text{$\mcO_{i}(t)$ and $V_{i}(t)$ are connected for all $t\in\R$},\\[2pt]
\text{$\mcO_{i}(t)\setminus\overbar{V_{i}(t)}$ is connected for all $t\in\R$},\\[2pt]
\text{$\partial\big(\mcO_{i}(t)\setminus\overbar{V_{i}(t)}\big)$ is smooth for all $t\in\R$}.
\end{array}\right.
\end{equation}

The proof of Theorem \ref{thma} is split into $5$ steps. We first extend $v$ to a global solution $\tilde{v}$ of the Stokes system with distributed forcing in Section \ref{Ext}. Then in Section \ref{App-disc} we establish that $\tilde{v}$ can be approximated on $K$, up to $\ep$, by a solution $v_{1}$ of the Stokes system with a forcing which corresponds to a finite number of Dirac masses outside of a space-time cylinder containing $K$. Next in Section \ref{Fourier-exp}, with intention to approximate $\tilde{v}$ by a global solution of the Stokes system without forcing, we expand $v_{1}$ into the series of the vector spherical harmonics. In Section \ref{Trunc}, by making a suitable truncation to this series expansion, we find an approximation $v_{2}$ of $v_{1}$ written as a finite sum of terms. In Section \ref{Global}, based on the ODE theory, we verify that $v_{2}$ can be extended to a global solution of the Stokes system without forcing, uniquely in the class \eqref{growth.thma}. Then, by gathering all the estimates for approximation, putting $u=v_{2}|_{\R^{4}_{+}}$ completes the proof of Theorem \ref{thma}.

    \subsection{Extension into a global solution with distributed forcing}
    \label{Ext}

The first step consists of extending the local solution $v$ of the homogeneous (without forcing) Stokes system \eqref{intro.eq.S} in $K$ to a global solution $\tilde{v}$ of the Stokes system with distributed forcing. Moreover, as this forcing is integrable and compactly supported, $\tilde{v}$ is given in terms of the fundamental solution of the Stokes system which we now recall.

We recall here the expression of the tensor Green's function associated with the Stokes system in the full system; see \cite[Chapter 4]{Lad1969book}. 
Define the tensor $\Gamma=(\Gamma^{1},\Gamma^{2},\Gamma^{3})$ by
\begin{align}\label{def.fund.sol.1}
\Gamma^{k}
=
-\Delta \gamma^{k} + \nabla(\nabla\cdot \gamma^{k}).
\end{align}
Here the vector $\gamma^{k}=\gamma^{k}(x,t)$ solves the system
\begin{align}\label{def.fund.sol.2}
\left\{
\begin{array}{ll}
(\partial_{t} - \Delta) \Delta \gamma^{k}
= \delta_{\{x=0\}} \delta_{\{t=0\}} {\bf e}_{k}
&\mbox{if}\ t>0,\\
\gamma^{k}=0
&\mbox{if}\ t<0,
\end{array}
\right.
\end{align}
where $\delta_{\{\cdot\}}$ is the Dirac mass supported in $\{\cdot\}$ and ${\bf e}_{k}^{j}=\delta_{j,k}$, with explicit representation
\begin{align}\label{def.fund.sol.3}
\gamma^{k}(x,t)
=
\bigg\{
\frac{1}{(4\pi)^{\frac{5}{2}} t^{\frac32}} 
\int_{\R^{3}} \frac{1}{|x-z|}
e^{-\frac{|z|^2}{4 t}} \dd z
\bigg\}
{\bf e}_{k},
\quad
t>0.
\end{align}
Moreover, define the vector $\Pi$ by 
\begin{align}\label{def.fund.sol.4}
\Pi(x)
= \nabla
\bigg(-\frac{1}{4\pi|x|}
\bigg).
\end{align}
Then, for given $f=f(x,t)$ with suitable decay at infinity, the pair
\begin{align*}
\begin{split}
v(x,t) 
&=
\int_{\R^{4}}
\Gamma(x-y,t-s) f(y,s) \dd y \dd s,\\
q(x,t) 
&=
\int_{\R^{3}}
\Pi(x-y) \cdot f(y,t) \dd y
\end{split}
\end{align*}
satisfies the Stokes system with forcing $f$:
\begin{equation*}
\left\{
\begin{array}{ll}
\partial_{t} v - \Delta v + \nabla q
= f&\mbox{in}\ \R^{4},\\
\nabla\cdot v
= 0&\mbox{in}\ \R^{4}.
\end{array}\right.
\end{equation*}
The solutions are unique in suitable classes thanks to the Liouville theorem \cite{JSS2012}.

Let $(v,q)$ be given as in Theorem \ref{thma} and recall \eqref{assum1.O}--\eqref{assum3.O}. Then we have the following.

\begin{lemma}\label{lem.Ext}
There is $\varphi \in L^{1}(\R^{4})^3$ supported in $B_{R}\times(0,T)$
such that
\begin{align}\label{rep.lem.Ext}
\begin{split}
\tilde{v}(x,t)
&=\int_{\R^{4}} \Gamma(x-y,t-s)\varphi(y,s)\dd y\dd s,\\
\tilde{q}(x,t)  &=\int_{\R^{3}}
\Pi(x-y) \cdot \varphi(y,t) \dd y
\end{split}
\end{align}
satisfies 
\begin{equation}\label{cons.lem.Ext}
(\tilde{v},\tilde{q})
=(v,q)
\quad \text{on} \mkern9mu K.
\end{equation}
\end{lemma}

\begin{proof}
Let $\chi=\chi(x,t)$  be a smooth cut-off function such that 
\[
\chi(x,t)=1
\quad \text{on} \mkern9mu \overbar{V}
\qquad \text{and} \qquad
\opsupp \chi \subset \mcO .
\]
We introduce a Bogovskii operator; see \cite{BorSoh1990} for example. Fix $t\in(0,T)$ and set 
\[
A_{i} = \mcO_{i}(t)\setminus\overbar{V_{i}(t)}.
\]
Then let $\opbog_{i}=\opbog_{i}[\,\cdot\,]$ denote a Bogovskii operator on the annular domain $A_{i}$ which maps $f\in L^{2}(A_{i})$ with $\int_{A_{i}} f=0$ to $\opbog_{i}[f]=:w$ belonging to $H^{1}(A_{i})^{3}$ and solving
\begin{equation*}
\left\{
\begin{array}{ll}
\nabla\cdot w=f&\mbox{in}\ A_{i},\\
w = 0&\mbox{on}\ \partial A_{i}.
\end{array}\right.
\end{equation*}
We will apply $\opbog_{i}$ on $A_{i}$ to
\[
f:=(\nabla\chi\cdot v)(\cdot,t)|_{A_{i}} \in L^{2}(A_{i}).
\]
The condition $\int_{A_{i}} f=0$ is verified as follows. By the Gauss divergence theorem, we have
\[
\int_{A_{i}} f
= \int_{A_{i}} \nabla\cdot(\chi v)(\cdot,t)
= \int_{\partial A_{i}} (\chi v)(\cdot,t)\cdot n \dd \sigma.
\]
Here $n$ denotes the outward unit normal vector to $\partial A_{i}$ and $\dd \sigma$ the surface measure of $A_{i}$. The definition of the cut-off function $\chi$ implies that
\[
\int_{\partial A_{i}} (\chi v)(\cdot,t)\cdot n \dd \sigma
= \int_{\partial V_{i}(t)} v(\cdot,t)\cdot n \dd \sigma.
\]
Since $V_{i}(t)$ is a connected open set in $\R^{3}$, by the Gauss divergence theorem again, we have
\[
\int_{\partial V_{i}(t)} v(\cdot,t)\cdot n \dd \sigma
= \int_{V_{i}(t)} \nabla\cdot v(\cdot,t)
= 0.
\]
Thus $\int_{A_{i}} f=0$ holds and $\opbog_{i}$ is applicable to $f=(\nabla\chi\cdot v)(\cdot,t)|_{A_{i}}$. Moreover, thanks to \cite[Theorem 2.4]{BorSoh1990}, $\opbog_{i}[(\nabla\chi\cdot v)(\cdot,t)|_{A_{i}}]$ is smooth and compactly supported in $A_{i}$.

Define
\[
\opbog[(\nabla\chi\cdot v)(\cdot,t)]
=\sum_{i=1}^{N} \opbog_{i}[(\nabla\chi\cdot v)(\cdot,t)|_{A_{i}}].
\]
Then we extend $\opbog[(\nabla\chi\cdot v)(\cdot,t)]$ to $\R^{3}$ by zero and then to $\R^{4}$ by moving $t\in\R$. This extension will be denoted by $\opbog[(\nabla\chi\cdot v)]$ again in the following. Notice that
\begin{align}\label{supp.bog.prf.lem.Ext}
\opsupp\opbog[(\nabla\chi\cdot v)(\cdot,t)]
\subset
\mcO\setminus\overbar{V}.
\end{align}

Now we set
\begin{align}\label{def.tilde.vq.prf.lem.Ext}
\tilde{v}
= \chi v - \opbog[\nabla\chi\cdot v]
\qquad \text{and} \qquad
\tilde{q}
= \chi q.
\end{align}
Then $(\tilde{v},\tilde{q})$ satisfies
\begin{equation}\label{eq.prf.lem.Ext}
\left\{
\begin{array}{ll}
\partial_{t} \tilde{v} - \Delta\tilde{v} + \nabla\tilde{q}
= \varphi&\mbox{in}\ \R^{4},\\
\nabla\cdot \tilde{v}
= 0&\mbox{in}\ \R^{4},\\
\tilde{v}
= 0&\mbox{on}\ \R^{3}\times\{0\}
\end{array}\right.
\end{equation}
with forcing $\varphi$ with support in $B_{R}\times(0,T)$ defined by
\begin{align*}
\varphi
= \big((\partial_{t} - \Delta)\chi\big)v
- 2\nabla\chi\cdot\nabla v
+ (\nabla\chi)q
- (\partial_{t} - \Delta)\opbog[\nabla\chi\cdot v].
\end{align*}
Moreover, $(\tilde{v},\tilde{q})$ is supported in $B_{R}\times(0,T)$ and \eqref{cons.lem.Ext} holds by \eqref{supp.bog.prf.lem.Ext}.

It remains to prove that $\varphi$ belongs to $L^{1}(\R^{4})^{3}$ and that the representation \eqref{rep.lem.Ext} holds. Estimates of the operator $\opbog$ in \cite[Theorem 2.4]{BorSoh1990} show that
\begin{align*}
\begin{split}
&\|(\partial_{t} - \Delta)\opbog[\nabla\chi\cdot v]\|_{L^{2}(B_{R}\times(0,T))}\\
&\lesssim
\|\opbog[\partial_{t}(\nabla\chi\cdot v)]\|_{L^{2}(B_{R}\times(0,T))}
+ \|\Delta\opbog[\nabla\chi\cdot v]\|_{L^{2}(B_{R}\times(0,T))}\\
&\lesssim
\|(v, \nabla v, \partial_{t} v)\|_{L^{2}(B_{R}\times(0,T))}.
\end{split}
\end{align*}
Thus we obtain
\begin{equation}
\label{esti-bog}
\|\varphi\|_{L^{1}(B_{R}\times(0,T))}
\lesssim
\|(v, \nabla v, \partial_{t} v, q)\|_{L^{2}(B_{R}\times(0,T))}.
\end{equation}
Since both $\tilde{v}$ and $\tilde{q}$ decay at spatial infinity by \eqref{def.tilde.vq.prf.lem.Ext}, the representation \eqref{rep.lem.Ext} follows from the Liouville theorem \cite{JSS2012} applied to \eqref{eq.prf.lem.Ext} with $\varphi=0$. This completes the proof.
\end{proof}

    \subsection{Approximation by  a global solution with discrete forcing}
    \label{App-disc}

In this second step, we will establish that the velocity field $\tilde{v}$ in Lemma \ref{lem.Ext} can be approximated on $K$, up to $\ep$, by a global solution of the Stokes system with forcing having only a finite number of poles, all outside of $K$. Recall the representation \eqref{rep.lem.Ext} and \eqref{assum1.O}--\eqref{assum3.O}. 

Let $L_{1},L_{2}\subset\R^{4}$ be open sets such that
\[
\opsupp\varphi \subset L_{1} \Subset \mcO
\qquad \text{and} \qquad
V \Subset \overbar{L_{1}}^{{\rm c}}
\]
and that 
\begin{equation*}
\left\{
\begin{array}{ll}
V \Subset L_{2} \subset L_{1}^{{\rm c}}\cap (B_{R}\times(0,T)),\\[2pt]
\opsupp\varphi \subset \overbar{L_{2}}^{{\rm c}},\\[2pt]
\text{$\overbar{L_{2}}^{{\rm c}}$ is a connected open set},\\[2pt]
\text{$\overbar{L_{2}}^{{\rm c}}(t)$ is connected for all $t\in\R$.}
\end{array}\right.
\end{equation*}

The following lemma corresponds to \cite[Lemma 3.3]{EGFPS2019}.

\begin{lemma}\label{lem.conv}
Let $\varphi\in L^1(\R^{4})^{3}$ be of compact support. For any $U\subset\R^{4}$ with $\opsupp\varphi\subset U$ and $\ep>0$, 
there are 
$I\in\N$,
$\{(z_{i},r_{i})\}_{i=1}^{I}\subset \opsupp\varphi$,
and $\{d_{i}\}_{i=1}^{I}\subset\R^3$ 
such that
\begin{align*}
\bigg\|
\int_{\R^{4}} \Gamma(\,\cdot\,-y,\,\cdot\,-s) \varphi(y,s)\dd y\dd s
- \sum_{i=1}^{I}  \Gamma(\,\cdot\,-z_{i},\,\cdot\,-r_{i}) d_{i}
\bigg\|_{C(U^{\rm c})} 
< \ep. 
\end{align*}
\end{lemma}

\begin{proof}
The proof is parallel to the one of \cite[Lemma 3.3]{EGFPS2019} and thus omitted here.
\end{proof}

Applying Lemma \ref{lem.conv} to $\tilde{v}$ defined in Lemma \ref{lem.Ext} with $U=L_{1}$, we have
\begin{align}\label{inqe1.App-disc}
\bigg\|\tilde{v}
- \sum_{i=1}^{I} \Gamma(\,\cdot\,-z_{i},\,\cdot\,-r_{i})d_{i}
\bigg\|_{C(L_{1}^{{\rm c}})}
<\ep
\end{align}
for some $I\in\N$,
$\{(z_{i},r_{i})\}_{i=1}^{I}\subset \opsupp\varphi$, and $\{d_{i}\}_{i=1}^{I}\subset\R^3$.

The following lemma corresponds to \cite[Lemma 3.2]{EGFPS2019}.

\begin{lemma}\label{lem.poles}
Let $U\subset\R^{4}$ be a connected open set such that $U(t)$ is connected for all $t\in\R$. Fix $(y,s)\in U$ and a bounded connected open set $M\subset U$ such that $M(s)\neq\emptyset$. For any $\ep>0$, 
there are $J\in\N$,
$\{(y_j,s_j)\}_{j=1}^{J}\subset M$,
and $\{b_j\}_{j=1}^{J}\subset\R$ such that 
\begin{align}\label{est.lem.poles}
\bigg\|
\Gamma(\,\cdot\,-y,\,\cdot\,-s)
- \sum_{j=1}^{J} b_j \Gamma(\,\cdot\,-y_j,\,\cdot\,-s_j)
\bigg\|_{C(U^{\rm c})} 
< \ep. 
\end{align}
\end{lemma}

\begin{proof}
Set $\tilde{\Gamma}=\Gamma(\,\cdot\,-y,\,\cdot\,-s)$. By taking a connected open set $\tilde{U} \subsetneq U$ such that
\begin{equation*}
\left\{
\begin{array}{ll}
(y,s)\in \tilde{U},\\[2pt]
M \subset \tilde{U},\\[2pt]
\text{$\tilde{U}(t)$ is connected for all $t\in\R$,}
\end{array}\right.
\end{equation*}
we consider the vector space
\begin{align*}
\mcV
= \opspan
\big\{
\Gamma(\,\cdot\,-z,\,\cdot\,-\tau)|_{\tilde{U}^{\rm c}}
~\big|~
(z,\tau)\in M
\big\}.
\end{align*}
We may regard $\mcV$ as a subspace of the Banach space $C_{\infty}(\tilde{U}^{\rm c})^{3\times3}$ where $C_{\infty}(\tilde{U}^{\rm c})$ refers to the space of continuous functions on $\tilde{U}^{\rm c}$ vanishing at infinity. Note that $\tilde{\Gamma}\in C_{\infty}(\tilde{U}^{\rm c})^{3\times3}$.

Then, by the Hahn–Banach theorem, it suffices to show that $L\tilde{\Gamma}=0$ holds for any continuous linear functional $L$ on $C_{\infty}(\tilde{U}^{\rm c})^{3\times3}$ which vanishes on $\mcV$. This fact implies that $\tilde{\Gamma}$ can be uniformly approximated by the elements of $\mcV$ in the topology of $C_{\infty}(\tilde{U}^{\rm c})^{3\times3}$.

Take such a functional $L$ on $C_{\infty}(\tilde{U}^{\rm c})^{3\times3}$ arbitrarily. By the Riesz–Markov–Kakutani representation theorem, the topological dual of $C_{\infty}(\tilde{U}^{\rm c})^{3\times3}$ is $\mathcal{M}(\tilde{U}^{\rm c})^{3\times3}$ where $\mathcal{M}(\tilde{U}^{\rm c})$ refers to the space of finite regular signed Borel measures on $\R^{4}$ supported in $\tilde{U}^{\rm c}$. Thus, for the given $L$, there is an element $\mu\in \mathcal{M}(\tilde{U}^{\rm c})^{3\times3}$ such that
\begin{align*}
L\Theta
= \int_{\tilde{U}^{\rm c}}
\Theta(x,t) \dd \mu(x,t),
\quad
\Theta\in C_{\infty}(\tilde{U}^{\rm c})^{3\times3}
\end{align*}
and
\begin{align}\label{onV.prf.lem.poles}
L\Theta
=
\int_{\tilde{U}^{\rm c}}
\Theta(x,t) \dd \mu(x,t) =0,
\quad
\Theta\in \mcV.
\end{align}
Let $\Gamma^{\ast}=\Gamma^{\ast}(x,t)$ be the fundamental solution of the adjoint Stokes system
\begin{equation*}
\left\{
\begin{array}{ll}
\partial_{t} v + \Delta v + \nabla q
= f&\mbox{in}\ \R^{4},\\
\nabla\cdot v
= 0&\mbox{in}\ \R^{4},
\end{array}\right.
\end{equation*}
which is given by $\Gamma^{\ast}(x,t)=\Gamma(-x,-t)$ explicitly. Define the tensor
\begin{align}\label{def.F.prf.lem.poles}
\begin{split}
F(z,\tau)
&=\int_{\tilde{U}^{\rm c}}
\Gamma^{\ast}(z-x,\tau-t)
\dd \mu(x,t)\\
&=\int_{\tilde{U}^{\rm c}}
\Gamma(x-z,t-\tau)
\dd \mu(x,t),
\quad
(z,\tau)\in \tilde{U}.
\end{split}
\end{align}
Notice that $F$ is well-defined and continuous at each $(z,\tau)\in \tilde{U}$ since integration is taken over $\tilde{U}^{\rm c}$. Moreover, we see that each column of the matrix $F$ is a solution of
\begin{equation}\label{ad.stokes.prf.lem.poles}
\left\{
\begin{array}{ll}
\partial_{\tau} v + \Delta_{z} v + \nabla_{z} q
= 0&\mbox{in}\ \tilde{U},\\
\nabla_{z}\cdot v
= 0&\mbox{in}\ \tilde{U}.
\end{array}\right.
\end{equation}
Then, by the definition of $F$ in \eqref{def.F.prf.lem.poles} and the condition \eqref{onV.prf.lem.poles}, we see that $F(z,\tau)=0$ for any $(z,\tau)\in M$. Since $M(s)\subset \tilde{U}(s)$ and $\tilde{U}(s)$ is connected, the unique continuation for \eqref{ad.stokes.prf.lem.poles} in space, or the unique continuation in \cite{FabLeb2002} for \eqref{ad.stokes.prf.lem.poles} in space-time, implies that $F$ vanishes on $\tilde{U}(s)$. In particular, at $(y,s)\in \tilde{U}(s)$,
\begin{align*}
0
=F(y,s)
=
\int_{\tilde{U}^{\rm c}}
\Gamma(x-y,t-s)
\dd \mu(x,t)
=
L\tilde{\Gamma}.
\end{align*}
Thus the assertion follows by the Hahn–Banach theorem. This completes the proof.
\end{proof}

Recall \eqref{inqe1.App-disc} and let $M\subset\R^{4}$ be an open set such that
\begin{equation*}
\left\{
\begin{array}{ll}
\text{$M$ is bounded and connected},\\[2pt]
\text{$M\subset \overbar{B_{R}\times(0,T)}^{{\rm c}}$},\\[2pt]
\text{$M(r_{i})\neq\emptyset$ for all $i\in I$.}
\end{array}\right.
\end{equation*}
Then we have the main result of this second step as follows.

\begin{lemma}\label{lem.App-disc}
There are 
$J\in\N$,
$\{(y_{j},s_{j})\}_{j=1}^{J}\subset M$,
and
$\{c_{j}\}_{j=1}^{J}\subset\R^{3}$
such that
\begin{align}\label{def.v1}
v_1(x,t)
&:= \sum_{j=1}^{J} \Gamma(x-y_j,t-s_j) c_{j}
\end{align}
approximates $\tilde{v}$ in Lemma \ref{lem.Ext} as 
\begin{equation}\label{capp}
\|\tilde{v} - v_1\|_{C(\overbar{L_{2}})}
< \ep.
\end{equation}
Moreover, $v_1$ is smooth on $B_R\times\R$ and satisfies 
\begin{equation}\label{est.v1}
\sup_{x\in B_R}
|\partial_{t}^k \partial_x^\alpha v_1(x,t)|
\lesssim_{\,\alpha,k}
(1+|t|)^{-\frac32-\frac{|\alpha|}{2}-k}
\end{equation}
\end{lemma}

\begin{proof}
Fix $1\le i\le I$ in \eqref{inqe1.App-disc}. Lemma \ref{lem.poles} with $U=\overbar{L_{2}}^{{\rm c}}$ and $(y,s)=(z_{i},s_{i})$ implies that
\begin{align}\label{inqe2.App-disc}
\bigg\|
\Gamma(\,\cdot\,-z_{i},\,\cdot\,-r_{i})
- \sum_{i=1}^{J_{i}}
b^{(i)}_{j}
\Gamma(\,\cdot\,-y^{(i)}_{j},\,\cdot\,-s^{(i)}_{j})
\bigg\|_{C(\overbar{L_{2}})}
< \frac{\ep}{I \max\{1,|d_{i}|\}}
\end{align}
for some 
$J_{i}\in\N$,
$\{(y^{(i)}_{j},s^{(i)}_{j})\}_{i=1}^{J_{i}}\subset M$,
and 
$\{b^{(i)}_{j}\}_{j=1}^{J_{i}}\subset\R$.

Gathering the two sums in \eqref{inqe1.App-disc} and \eqref{inqe2.App-disc} and rearranging them into a single one, we have the right of \eqref{def.v1} and \eqref{capp}. Finally, the estimate \eqref{est.v1} is obtained by the definition of $\Gamma$ in \eqref{def.fund.sol.1}--\eqref{def.fund.sol.3} and observing that each pole $(y_{j},s_{j})$ is located outside $B_R\times\R$.
\end{proof}

    \subsection{Fourier transform in time and series expansion}
    \label{Fourier-exp}

Our objective in the subsequent steps is to approximate $v_{1}$ in Lemma \ref{lem.App-disc} by the vector fields defined on the whole space-time $\R^{4}$, thereby completing the proof of Theorem \ref{thma}. These global approximations must be built to solve the Stokes system in $\R^{4}$ without forcing. To accomplish this, we will utilize the Fourier transformation of $v_{1}$ in time and the series expansion in vector harmonics. Here we may take a truncation both in time frequencies $\tau$ and the modes of series $l$; see Lemma \ref{lem.Trunc} in the next step. This lemma allows us to avoid complicated analysis when $|\tau|$ is near zero or infinity and $l$ is large in Section \ref{Global}.

Fix $0<\rho<R$. Let us introduce the Fourier transform in time of $v_{1}$
\begin{align}\label{def.fourier.v1}
\begin{split}
\hat{v}_{1}(x,\tau)
=
\int_{-\infty}^{\infty}
v_{1}(x,t) e^{-\ii t\tau} \dd t,
\quad (x,\tau)\in B_{\rho}\times\R.
\end{split}
\end{align}
Integration by parts and the estimate \eqref{est.v1} implies that, for $k\in\Z_{\ge0}$,
\begin{align}\label{est.fourier.v1}
\sup_{x\in B_R}
|\partial_x^\alpha \hat{v}_{1}(x,\tau)|
\lesssim_{\,k,\alpha,v_{1}}
(1+|\tau|)^{-k},
\quad (x,\tau)\in B_{\rho}\times\R.
\end{align}
Then it is easy to check the well-definedness of the inverse Fourier transformation
\begin{align}\label{def.inv.fourier.v1}
\begin{split}
v_{1}(x,t)
&=
\frac{1}{2\pi}
\int_{-\infty}^{\infty}
\hat{v}_{1}(x,\tau) e^{\ii \tau t} \dd \tau,
\quad (x,t)\in B_{\rho}\times\R.
\end{split}
\end{align}
The transform $\hat{v}_{1}$ is expended into the vector spherical harmonics in Section \ref{sec.SH} as
\begin{align}\label{exp.v1}
\begin{split}
\hat{v}_{1}(x,\tau)
&=
\hat{v}_{1}(r,\theta,\phi,\tau)\\
&=
\sum_{l=1}^{\infty} \sum_{m=-l}^{l}
\big(
c^{r}_{lm}(r,\tau) \bm{Y}_{lm}
+ c^{(1)}_{lm}(r,\tau) \bm{\Psi}_{lm}
+ c^{(2)}_{lm}(r,\tau) \bm{\Phi}_{lm}
\big).
\end{split}
\end{align}
These coefficients $c^{r}_{lm},c^{(1)}_{lm},c^{(2)}_{lm}$ are defined by \eqref{coef.exp.S.VSH} with $w^{r}_{lm},w^{(1)}_{lm},w^{(2)}_{lm}$ and $w$ replaced by $c^{r}_{lm},c^{(1)}_{lm},c^{(2)}_{lm}$ and $\hat{v}_{1}$, respectively. The series converges absolutely and uniformly on $B_{\rho}\times\R$ thanks to Lemma \ref{lem.VSH.exp} (\ref{item2.lem.VSH.exp}). Observe that the terms with $l=0$ are absent in \eqref{exp.v1}. This is because of \eqref{Psi.Phi.00} and the condition $c^{r}_{00}=0$. The latter follows from
\[
\bm{Y}_{00}
= \bm{\hat r} Y_{00}
= \bm{\hat r} \sqrt{\frac{1}{4\pi}}
\]
and the Gauss divergence theorem
\[
c^{r}_{00}
=
\langle
\hat{v}_{1}, \bm{Y}_{00}
\rangle_{S}
=
\sqrt{\frac{1}{4\pi}}
\frac{1}{r^{2}}
\int_{S_{r}}
\hat{v}_{1}\cdot \bm{\hat r}
=
\sqrt{\frac{1}{4\pi}}
\frac{1}{r^{2}}
\int_{B_{r}}
\nabla\cdot \hat{v}_{1}
=
0,
\quad 0<r<\rho.
\]
Consequently, from \eqref{def.inv.fourier.v1}--\eqref{exp.v1}, we obtain the following representation of $v_{1}$: 
\begin{align}\label{exp2.v1}
\begin{split}
v_{1}(x,t)
&=v_{1}(r,\theta,\phi,t)\\
&=
\frac{1}{2\pi}
\sum_{l=1}^{\infty} \sum_{m=-l}^{l}\\
&\qquad\quad
\int_{-\infty}^{\infty}
\big(
c^{r}_{lm}(r,\tau) \bm{Y}_{lm}
+ c^{(1)}_{lm}(r,\tau) \bm{\Psi}_{lm}
+ c^{(2)}_{lm}(r,\tau) \bm{\Phi}_{lm}
\big)
e^{\ii \tau t}
\dd \tau.
\end{split}
\end{align}

    \subsection{Truncation}
    \label{Trunc}

The following is a useful lemma for constructing global approximations.

\begin{lemma}\label{lem.Trunc}
For any $\ep>0$, there are $L\in\Z_{\ge0}$ and $0<\tau_{1}<\tau_{2}<\infty$ such that the truncation of \eqref{exp2.v1}
\begin{align}\label{def.v2}
\begin{split}
v_{2}(x,t)
&=v_{2}(r,\theta,\phi,t)\\
&:=
\frac{1}{2\pi}
\sum_{l=1}^{L} \sum_{m=-l}^{l}\\
&\qquad\quad
\int_{\tau_{1} < |\tau| < \tau_{2}}
\big(
c^{r}_{lm}(r,\tau) \bm{Y}_{lm}
+ c^{(1)}_{lm}(r,\tau) \bm{\Psi}_{lm}
+ c^{(2)}_{lm}(r,\tau) \bm{\Phi}_{lm}
\big)
e^{\ii \tau t}
\dd \tau
\end{split}
\end{align}
satisfies
\begin{align}\label{ineq.lem.Trunc}
\|v_{1} - v_{2}\|_{C(B_{\rho}\times\R)}
< \ep. 
\end{align}
\end{lemma}

\begin{proof}
Let $\ep>0$. Thanks to \eqref{est.fourier.v1} with $\alpha=0$, there are $0<\tau_{1}<\tau_{2}<\infty$ such that
\begin{align*}
\begin{split}
\bigg|
v_{1}(x,t)
- \frac{1}{2\pi}
\int_{\tau_{1} < |\tau| < \tau_{2}}
\hat{v}_{1}(x,\tau) e^{\ii \tau t} \dd \tau
\bigg|
<\frac{\ep}2.
\end{split}
\end{align*}
Then, by the convergence of the series in \eqref{exp.v1}, there is $L\in\Z_{\ge0}$ such that
\begin{align*}
\begin{split}
\bigg|
\hat{v}_{1}(x,\tau)
- \sum_{l=1}^{L} \sum_{m=-l}^{l}
\big(
c^{r}_{lm} \bm{Y}_{lm}
+ c^{(1)}_{lm} \bm{\Psi}_{lm}
+ c^{(2)}_{lm} \bm{\Phi}_{lm}
\big)
\bigg|
<
\frac{\pi}{2(\tau_{2}-\tau_{1})} \ep.
\end{split}
\end{align*}
The assertion \eqref{ineq.lem.Trunc} is a consequence of these two estimates.
\end{proof}

   \subsection{Global approximation and end of the proof}
   \label{Global}

Regarding the coefficients in \eqref{exp.v1}
\[
c^{r}_{lm}(r,\tau),
\quad 
c^{(1)}_{lm}(r,\tau),
\quad
c^{(2)}_{lm}(r,\tau)
\]
as functions of $r\in[0,\rho)$ after fixing $\tau\in\R$, we will see in the proof of Lemma \ref{lem.coeff.v1} below that they satisfy a certain system of ODEs having singularities only at $r=0$. The solutions of such ODEs can be extended outside the finite interval $[0,\rho)$ and to $[0,\infty)$ by general theory. This enables us to extend the approximation $v_{2}$ in Lemma \ref{lem.Trunc} to the whole space-time $\R^{4}$ and hence to obtain a global approximation of $v_{1}$ thanks to the truncation.

Therefore, our main task is to estimate the coefficients $c^{r}_{lm}, c^{(1)}_{lm}, c^{(2)}_{lm}$ in \eqref{exp.v1} which are considered as functions of $r\in [0,\infty)$. Here we can take advantage of the fact that, according to Lemma \ref{lem.Trunc}, the time frequency $\tau$ can be assumed to satisfy $0<\tau_{1}\le|\tau|\le\tau_{2}<\infty$. Once $c^{r}_{lm}, c^{(1)}_{lm}, c^{(2)}_{lm}$ are estimated, the proof of Theorem \ref{thma} will be completed.

\textit{In the rest of this section, we regard a function $f(r,\tau)$ depending on both $r$ and $\tau$ as a function of $r$ if $\tau$ is fixed. An ``ODE" satisfied by $f(r,\tau)$ is understood in this convention.}

Observe that $v_1$ satisfies, with $0<\rho<R$ arbitrarily fixed, 
\begin{equation}\label{eq.v1q1.prf.thma}
\left\{
\begin{array}{ll}
\partial_{t} v_1 - \Delta v_1 + \nabla q_{1}
= 0&\mbox{in}\ B_{\rho}\times\R,\\
\nabla\cdot v_1
= 0&\mbox{in}\ B_{\rho}\times\R
\end{array}\right.
\end{equation}
for some smooth pressure $q_1$. Then, by the Fourier transform \eqref{def.fourier.v1} and 
\begin{align}\label{def.lap.q1}
\begin{split}
\hat{q}_{1}(x,\tau)
&=
\int_{-\infty}^{\infty}
q_{1}(x,t) e^{-\ii t\tau} \dd t,
\quad (x,\tau)\in B_{\rho}\times\R,
\end{split}
\end{align}
we see that the pair $(\hat{v}_{1},\hat{q}_{1})$ satisfies the vector Helmholtz equation with forcing
\begin{equation}\label{eq.helm}
\left\{
\begin{array}{ll}
\ii\tau \hat{v}_{1} - \Delta \hat{v}_{1} 
= -\nabla \hat{q}_{1}&\mbox{in}\ B_{\rho}\times\R,\\
\nabla\cdot \hat{v}_{1}
= 0&\mbox{in}\ B_{\rho}\times\R.
\end{array}\right.
\end{equation}

Before considering the solution $\hat{v}_{1}$ of \eqref{eq.helm}, we give estimates of the forcing $-\nabla \hat{q}_{1}$. Let us expand $\hat{q}_{1}$ into the (scalar) spherical harmonics in Section \ref{sec.SH} as
\begin{align}\label{exp.q1}
\hat{q}_{1}(x,\tau)
=
\hat{q}_{1}(r,\theta,\phi,\tau)
=
\sum_{l=1}^{\infty} \sum_{m=-l}^{l}
b_{lm}(r,\tau) Y_{lm}.
\end{align}
Notice that we have taken $b_{00}=0$ in \eqref{exp.q1} without loss of generality.

\begin{lemma}\label{lem.coeff.q1}
The coefficient in \eqref{exp.q1} can be defined for $r\in(0,\infty)$ and satisfies
\begin{align}\label{est.coeff.q1}
|b_{lm}(r,\tau)|
\lesssim_{\,\tau_{2},l,\rho}
(1+r)^{l},
\quad r\in(0,\infty)
\end{align}
uniformly in $\tau\in\R$ with $0<\tau_{1}\le|\tau|\le\tau_{2}<\infty$.
\end{lemma}

\begin{proof}
Fix $\tau\in\R$. Operating $\nabla\cdot$ to the first line of \eqref{eq.helm}, we deduce that $\hat{q}_{1}=\hat{q}_{1}(x,\tau)$ solves $-\Delta \hat{q}_{1}=0$ in $B_{\rho}\times\R$. Then, substituting the series expansion \eqref{exp.q1} to it and using the relations \eqref{DS}--\eqref{DpasS} as well as Lemma \ref{lem.SH.exp}, we find that the coefficient $b_{lm}(r)$ satisfies
\begin{align}\label{eq2.q1.prf.thma}
\sum_{l=1}^{\infty} \sum_{m=-l}^{l}
\bigg\{
\Big(-\frac{\dd^2}{\dd r^2} 
- \frac{2}{r} \frac{\dd}{\dd r}
+ \frac{l(l+1)}{r^2}
\Big)
b_{lm}(r)
\bigg\}
Y_{lm}
=0,
\quad
0<r<\infty.
\end{align}
Since linearly independent solutions of
\[
-\frac{\dd^2 y}{\dd r^2} 
- \frac{2}{r} \frac{\dd y}{\dd r}
+ \frac{l(l+1)}{r^2} y
=0, 
\quad
0<r<\infty
\]
are
$r^{l}$ and
$r^{-l-1}$, 
we see that $b_{lm}(r)$ bounded near $r=0$ is given by
\begin{align}\label{rep.b}
b_{lm}(r) = B_{lm} r^{l}.
\end{align}
To determine $B_{lm}$, we substitute \eqref{rep.b} to \eqref{exp.q1} and then have
\[
(\nabla q)\cdot \bm{\hat r}
= \frac{\partial \hat{q}_{1}}{\partial r}
= \sum_{l=1}^{\infty} \sum_{m=-l}^{l}
lB_{lm} r^{l-1} Y_{lm}.
\]
Using the orthogonality \eqref{sortho} and the first line of \eqref{eq.helm}, we observe that
\begin{align*}
lB_{lm} r^{l-1}
=
\langle (\nabla q)\cdot \bm{\hat r}, Y_{lm} \rangle_{S}(r)
=
\langle (-\ii\tau \hat{v}_{1}+\Delta \hat{v}_{1})\cdot \bm{\hat r}, Y_{lm} \rangle_{S}(r).
\end{align*}
Integrating the both sides over $(0,\rho)$, we find that
\begin{align}\label{def.B}
B_{lm}
=
\rho^{-l}
\int_{0}^{\rho}
\langle (-\ii\tau \hat{v}_{1}+\Delta \hat{v}_{1})\cdot \bm{\hat r}, Y_{lm} \rangle_{S}(r)
\dd r.
\end{align}
Then the estimate \eqref{est.fourier.v1} gives, uniformly in $\tau\in\R$ with $0<\tau_{1}\le|\tau|\le\tau_{2}<\infty$,
\begin{align}\label{est.B}
|B_{lm}|
\lesssim_{\,l,\rho}
|\tau|
\lesssim_{\,\tau_{2},l,\rho}
1.
\end{align}
The assertion \eqref{est.coeff.q1} is a consequence of \eqref{rep.b} and \eqref{est.B}.
\end{proof}

We then consider the solution $\hat{v}_{1}$ of \eqref{eq.helm} expanded as in \eqref{exp.v1}.

\begin{lemma}\label{lem.coeff.v1}
The coefficients in \eqref{exp.v1} can be defined for $r\in(0,\infty)$ and satisfies
\begin{align}\label{est.coeff.v1}
|c^{r}_{lm}(r,\tau)|
+ |c^{(1)}_{lm}(r,\tau)|
+ |c^{(2)}_{lm}(r,\tau)|
\lesssim_{\,v_{1},\tau_{1},\tau_{2},l,\rho}
e^{\sqrt{\tau_{2}}r},
\quad r\in(0,\infty)
\end{align}
uniformly in $\tau\in\R$ with $0<\tau_{1}\le|\tau|\le\tau_{2}<\infty$.
\end{lemma}

\begin{proofx}{Lemma \ref{lem.coeff.v1} (first half)}
In the first half of the proof, we derive an explicit representation formula for the coefficients $c^{r}_{lm}, c^{(1)}_{lm}, c^{(2)}_{lm}$ by considering the system of ODEs that govern them. The formulas for the vector spherical harmonics in Section \ref{sec.SH} are applied. Similar calculations to those shown below can be found in \cite[Apendix A]{FLRO2022}.

Fix $\tau\in\R$ and recall $\mu_{l}$ defined in \eqref{def.mu}. Applying \eqref{nabla.VSH}, we have
\begin{align*}
\nabla \hat{q}_{1}
= \sum_{l=1}^{\infty} \sum_{m=-l}^{l}
\Big(
\dfrac{\dd b_{lm}}{\dd r}\bm{Y}_{lm}
+ \frac{b_{lm}}{r} \bm{\Psi}_{lm}
\Big).
\end{align*}
Applying \eqref{rot.VSH} twice, we have
\begin{align*}
&\nabla\times\nabla\times(
c^{r}_{lm} \bm{Y}_{lm}
+ c^{(1)}_{lm} \bm{\Psi}_{lm}
+ c^{(2)}_{lm} \bm{\Phi}_{lm}
)\\
&=
\Big\{
- \dfrac{\mu_{l}}{r} 
\Big(
\dfrac{\dd c^{(1)}_{lm}}{\dd r}
+ \dfrac{c^{(1)}_{lm}}{r}
- \dfrac{c^{r}_{lm}}{r}
\Big)
\Big\}
\bm{Y}_{lm}\\
&\quad
+ \Big\{
- \Big(
\dfrac{\dd}{\dd r}
+ \dfrac{1}{r}
\Big)
\Big(
\dfrac{\dd c^{(1)}_{lm}}{\dd r}
+ \dfrac{c^{(1)}_{lm}}{r}
+ \dfrac{c^{r}_{lm}}{r}
\Big)
\Big\}
\bm{\Psi}_{lm}\\
&\quad
+ \Big\{
\dfrac{\mu_{l}}{r^{2}} c^{(2)}_{lm}
- \Big(
\dfrac{\dd}{\dd r}
+ \dfrac{1}{r}
\Big)
\Big(
\dfrac{\dd c^{(2)}_{lm}}{\dd r}
+ \dfrac{c^{(2)}_{lm}}{r}
\Big)
\Big\}
\bm{\Phi}_{lm}.
\end{align*}
Thus, using the identity for smooth vector fields
\[
-\Delta w = \nabla\times\nabla\times w - \nabla (\nabla\cdot w)
\]
and the condition $\nabla\cdot \hat{v}_1=0$, we see that the first line of \eqref{eq.helm} is written as
\begin{align*}
0
&=
\ii\tau \hat{v}_1 - \Delta \hat{v}_1 + \nabla \hat{q}_{1}\\
&=
\ii\tau \hat{v}_1 + \nabla\times\nabla\times\hat{v}_1 + \nabla \hat{q}_{1}\\
&=
\sum_{l=1}^{\infty} \sum_{m=-l}^{l}
\bigg[
\Big\{
\ii \tau c^{r}_{lm}
- \dfrac{\mu_{l}}{r} 
\Big(
\dfrac{\dd c^{(1)}_{lm}}{\dd r}
+ \dfrac{c^{(1)}_{lm}}{r}
- \dfrac{c^{r}_{lm}}{r}
\Big)
+ \dfrac{\dd b_{lm}}{\dd r}
\Big\}
\bm{Y}_{lm}\\
&\qquad\qquad\qquad
+ \Big\{
\ii \tau c^{(1)}_{lm}
- \Big(
\dfrac{\dd}{\dd r}
+ \dfrac{1}{r}
\Big)
\Big(
\dfrac{\dd c^{(1)}_{lm}}{\dd r}
+ \dfrac{c^{(1)}_{lm}}{r}
+ \dfrac{c^{r}_{lm}}{r}
\Big)
+ \dfrac{b_{lm}}{r}
\Big\}
\bm{\Psi}_{lm}\\
&\qquad\qquad\qquad
+ \Big\{
\ii \tau c^{(2)}_{lm}
+ \dfrac{\mu_{l}}{r^{2}} c^{(2)}_{lm}
- \Big(
\dfrac{\dd}{\dd r}
+ \dfrac{1}{r}
\Big)
\Big(
\dfrac{\dd c^{(2)}_{lm}}{\dd r}
+ \dfrac{c^{(2)}_{lm}}{r}
\Big)
\Big\}
\bm{\Phi}_{lm}
\bigg].
\end{align*}
Applying \eqref{div.VSH}, we see that the second line of \eqref{eq.helm} is written as
\begin{align*}
0
=
\nabla\cdot \hat{v}_{1}
=
\sum_{l=1}^{\infty} \sum_{m=-l}^{l}
\Big(
\dfrac{\dd c^{r}_{lm}}{\dd r}
+ \dfrac{2c^{r}_{lm}}{r}
- \frac{\mu_{l} c^{(1)}_{lm}}{r}
\Big)
Y_{lm}.
\end{align*}
Then, after rearrangements, we find that $(c^{r}_{lm}, c^{(1)}_{lm})$ satisfies the system of ODEs
\begin{equation}\label{3odes.prf.thma}
\left\{
\begin{array}{ll}
-\dfrac{\dd}{\dd r} (rc^{(1)}_{lm})
+ c^{r}_{lm}
=
\dfrac{1}{\mu_{l}}
\Big\{
-\ii \tau (r^{2}c^{r}_{lm})
- r^{2}\dfrac{\dd b_{lm}}{\dd r}
\Big\}
&\mbox{in}\ (0,\infty),\\[10pt]
\dfrac{\dd}{\dd r}
\Big\{
-\dfrac{\dd}{\dd r} (rc^{(1)}_{lm})
+ c^{r}_{lm}
\Big\}
=
-\ii \tau (r c^{(1)}_{lm})
- b_{lm}
&\mbox{in}\ (0,\infty),\\[10pt]
\dfrac{\dd}{\dd r} (r^{2}c^{r}_{lm})
- \mu_{l} (rc^{(1)}_{lm})
= 0
&\mbox{in}\ (0,\infty)
\end{array}\right.
\end{equation}
and $c^{(2)}_{lm}$ satisfies the ODE
\begin{equation}\label{ode.prf.thma}
-\dfrac{\dd^2}{\dd r^2} (rc^{(2)}_{lm})
+ \Big(
\ii \tau
+ \dfrac{\mu_{l}}{r^{2}}
\Big) (rc^{(2)}_{lm})
=
0
\quad 
\mbox{in}\ (0,\infty).
\end{equation}

We claim that the system \eqref{3odes.prf.thma} is not overdetermined and equivalent with
\begin{equation}\label{2odes.prf.thma}
\left\{
\begin{array}{ll}
-\dfrac{\dd}{\dd r} (rc^{(1)}_{lm})
+ c^{r}_{lm}
=
\dfrac{1}{\mu_{l}}
\Big\{
-\ii \tau (r^{2}c^{r}_{lm})
- r^{2}\dfrac{\dd b_{lm}}{\dd r}
\Big\}
&\mbox{in}\ (0,\infty),\\[10pt]
\dfrac{\dd}{\dd r} (r^{2}c^{r}_{lm})
- \mu_{l} (rc^{(1)}_{lm})
= 0
&\mbox{in}\ (0,\infty).
\end{array}\right.
\end{equation}
Indeed, suppose that $(c^{r}_{lm}, c^{(1)}_{lm})$ is a solution of \eqref{2odes.prf.thma}. Differentiating the first line, we have
\begin{align}\label{eq1.equiv.prf.thma}
\dfrac{\dd}{\dd r}
\Big\{
-\dfrac{\dd}{\dd r} (rc^{(1)}_{lm})
+ c^{r}_{lm}
\Big\}
&=
\dfrac{1}{\mu_{l}}
\Big\{
-\ii \tau \dfrac{\dd}{\dd r}(r^{2}c^{r}_{lm})
- \dfrac{\dd}{\dd r}\Big(
r^{2}\dfrac{\dd b_{lm}}{\dd r}
\Big)
\Big\}.
\end{align}
Using the fact that $b_{lm}=b_{lm}(r)$ satisfies 
\begin{align*}
-\dfrac{\dd^{2} b_{lm}}{\dd r^{2}}
- \dfrac{2}{r}\dfrac{\dd b_{lm}}{\dd r}
+ \frac{\mu_{l}}{r^{2}} b_{lm}
=0
\end{align*}
as well as the second line of \eqref{2odes.prf.thma}, we rewrite the right of \eqref{eq1.equiv.prf.thma} as
\begin{align}\label{eq2.equiv.prf.thma}
\dfrac{1}{\mu_{l}}
\Big\{
-\ii \tau \dfrac{\dd}{\dd r}(r^{2}c^{r}_{lm})
- \dfrac{\dd}{\dd r}\Big(
r^{2}\dfrac{\dd b_{lm}}{\dd r}
\Big)
\Big\}
=
-\ii \tau (rc^{(1)}_{lm})
- b_{lm}.
\end{align}
From \eqref{eq1.equiv.prf.thma}--\eqref{eq2.equiv.prf.thma}, we see that the two systems \eqref{3odes.prf.thma} and \eqref{2odes.prf.thma} are equivalent.

To represent $c^{r}_{lm},c^{(1)}_{lm},c^{(2)}_{lm}$ explicitly, we solve \eqref{ode.prf.thma}--\eqref{2odes.prf.thma}. Let us first solve \eqref{2odes.prf.thma}. Differentiating the second line of \eqref{2odes.prf.thma} and using the first line, we have
\begin{align}\label{eq1.rep.prf.thma}
-\dfrac{\dd^{2}}{\dd r^{2}} (r^{2}c^{r}_{lm})
=
-\mu_{l} \dfrac{\dd}{\dd r} (rc^{(1)}_{lm})
=
-\mu_{l} c^{r}_{lm}
+ \Big\{
-\ii \tau (r^{2}c^{r}_{lm})
- r^{2}\dfrac{\dd b_{lm}}{\dd r}
\Big\}.
\end{align}
Setting $y(r)=r^{2}c^{r}_{lm}(r)$, we see that $y$ solves the ODE
\begin{align}\label{eq2.app.suppl}
-\frac{\dd^2 y}{\dd r^2}
+ \Big(
\ii \tau + \dfrac{\mu_{l}}{r^{2}}
\Big) y
=-r^{2} \dfrac{\dd b_{lm}}{\dd r},
\quad
0<r<\infty.
\end{align}
We recall that linearly independent solutions of the homogeneous equation of \eqref{eq2.app.suppl} are 
\begin{align*}
r^{\frac12} I_{l+\frac12}(\sqrt{\ii \tau} r)
\qquad \text{and} \qquad
r^{\frac12} K_{l+\frac12}(\sqrt{\ii \tau} r)
\end{align*}
with the Wronskian equal to $1$. Here $I_\nu(z), K_\nu(z)$ are the modified Bessel functions in Section \ref{MBF}. Thus we find that $c^{r}_{lm}=c^{r}_{lm}(r)$ bounded near $r=0$ is given by
\begin{align}\label{rep.cr}
c^{r}_{lm}(r)
= C^{r}_{lm} r^{-\frac32} I_{l+\frac12}(\sqrt{\ii \tau} r)
+ \Lambda_{lm}(r),
\quad
0<r<\infty.
\end{align}
Here $C^{r}_{lm}$ is some constant to be determined later and $\Lambda_{lm}(r)$ is defined by
\begin{align}\label{def.Lambda}
\begin{split}
\Lambda_{lm}(r)
&= 
-\int_{0}^{r} 
r^{-\frac32} K_{l+\frac12}(\sqrt{\ii \tau} r)
s^{\frac52} I_{l+\frac12}(\sqrt{\ii \tau} s)
\dfrac{\dd \hat{b}_{lm}}{\dd s}(s) \dd s\\
&\quad
- \int_{r}^{\infty}
r^{-\frac32} I_{l+\frac12}(\sqrt{\ii \tau} r)
s^{\frac52} K_{l+\frac12}(\sqrt{\ii \tau} s) 
\dfrac{\dd \hat{b}_{lm}}{\dd s}(s) \dd s\\
&= 
-lB_{lm}
\int_{0}^{r} 
r^{-\frac32} K_{l+\frac12}(\sqrt{\ii \tau} r)
s^{l+\frac32} I_{l+\frac12}(\sqrt{\ii \tau} s)
\dd s\\
&\quad
- lB_{lm}
\int_{r}^{\infty}
r^{-\frac32} I_{l+\frac12}(\sqrt{\ii \tau} r)
s^{l+\frac32} K_{l+\frac12}(\sqrt{\ii \tau} s) 
\dd s.
\end{split}
\end{align}
Then $c^{(1)}_{lm}$ is determined by the second line of \eqref{2odes.prf.thma} as
\begin{align}\label{rep.c1}
c^{(1)}_{lm}(r)
=
\frac{1}{\mu_{l} r}
\frac{\dd}{\dd r} (r^{2}c^{r}_{lm})
=
\frac{1}{\mu_{l}}
\Big(
2 c^{r}_{lm}
+ r \frac{\dd c^{r}_{lm}}{\dd r}
\Big),
\quad
0<r<\infty.
\end{align}
Next we solve \eqref{ode.prf.thma}. Putting $y(r)=rc^{(2)}_{lm}(r)$, we see that $y$ solves the homogeneous equation of \eqref{eq1.rep.prf.thma}. Thus $c^{(2)}_{lm}=c^{(2)}_{lm}(r)$ bounded near $r=0$ is given by
\begin{align}\label{rep.c2}
c^{(2)}_{lm}(r)
= C^{(2)}_{lm} r^{-\frac12} I_{l+\frac12}(\sqrt{\ii \tau} r),
\quad
0<r<\infty
\end{align}
with some constant $C^{(2)}_{lm}$ to be determined later.

We determine the constants $C^{r}_{lm},C^{(2)}_{lm}$ as follows. This approach is inspired by the proof of \cite[Theorem 2.4]{EncPer-Sal2021}. We first consider $C^{r}_{lm}$ in \eqref{rep.cr}. Multiplying the both sides by
\[
r^{-\frac32} \overbar{I_{l+\frac12}(\sqrt{\ii \tau} r)} r^{2}
\]
and integrating over $(0,\rho)$, we find
\begin{align}\label{def.Cr}
C^{r}_{lm}
=
\frac{1}{\mathcal{I}_{l+\frac12,-1}(\sqrt{\ii \tau})}
\int_{0}^{\rho}
(c^{r}_{lm}(r) - \Lambda_{lm}(r))
r^{-\frac32} \overbar{I_{l+\frac12}(\sqrt{\ii \tau} r)}
r^{2} \dd r.
\end{align}
Here we have defined
\[
\mathcal{I}_{l+\frac12,j}(\sqrt{\ii \tau})
=
\int_{0}^{\rho} 
r^{j} |I_{l+\frac12}(\sqrt{\ii \tau} r)|^{2}
\dd r
\]
for an admissible $j\in\Z$. In the similar manner, $C^{(2)}_{lm}$ in \eqref{rep.c2} is determined as
\begin{align}\label{def.C2}
C^{(2)}_{lm}
=
\frac{1}{\mathcal{I}_{l+\frac12,1}(\sqrt{\ii \tau})}
\int_{0}^{\rho}
c^{(2)}_{lm}(r)
r^{-\frac12} \overbar{I_{l+\frac12}(\sqrt{\ii \tau} r)}
r^{2} \dd r.
\end{align}
Thus the desired formulas \eqref{rep.cr}, \eqref{rep.c1} and \eqref{rep.c2} with \eqref{def.Cr}--\eqref{def.C2} are obtained.
\end{proofx}

Before continuing with the proof of Lemma \ref{lem.coeff.v1}, let us give the estimates for \eqref{rep.cr}, \eqref{rep.c1} and \eqref{rep.c2} as well as for \eqref{def.Cr}--\eqref{def.C2} in the following two lemmas.

\begin{lemma}\label{lem.est.crc1c2}
We have
\begin{align*}
\begin{split}
&|c^{r}_{lm}(r,\tau)|
+ |c^{(1)}_{lm}(r,\tau)|
+ |c^{(2)}_{lm}(r,\tau)|\\
&\lesssim_{\,\tau_{1},\tau_{2},l}
(|C^{r}_{lm}(\tau)|
+ |C^{(2)}_{lm}(\tau)|
+ 1)
e^{\sqrt{\tau_{2}}r},
\quad r\in(0,\infty)
\end{split}
\end{align*}
uniformly in $\tau\in\R$ with $0<\tau_{1}\le|\tau|\le\tau_{2}<\infty$.
\end{lemma}

\begin{proof}
Let us first consider $c^{r}_{lm}(r,\tau)$ defined in \eqref{rep.cr}. By Lemma \ref{lem.MBF}, we compute
\begin{align*}
\begin{split}
|r^{-\frac32} I_{l+\frac12}(\sqrt{\ii \tau} r)|
&\lesssim_{\,l}
\left\{
\begin{array}{ll}
|\tau|^{\frac{l}2+\frac14}
r^{l-1}
&\mbox{if}\
r \le (\Re\sqrt{\ii \tau})^{-1},\\[5pt]
|\tau|^{-\frac14}
r^{-2}
e^{\Re(\sqrt{\ii \tau})r} 
&\mbox{if}\ 
r\ge(\Re\sqrt{\ii \tau})^{-1}
\end{array}\right.\\[5pt]
&\lesssim_{\,l}
\left\{
\begin{array}{ll}
|\tau|^{\frac{l}2+\frac14}
r^{l-1}
&\mbox{if}\
r \le (\Re\sqrt{\ii \tau})^{-1},\\[5pt]
|\tau|^{\frac34}
e^{\Re(\sqrt{\ii \tau})r} 
&\mbox{if}\ 
r\ge(\Re\sqrt{\ii \tau})^{-1}
\end{array}\right.\\[5pt]
&\lesssim_{\,l}
\left\{
\begin{array}{ll}
\tau_{2}^{\frac{l}{2}+\frac{1}{4}}
r^{l-1}
&\mbox{if}\
r \le (\Re\sqrt{\ii \tau})^{-1},\\[5pt]
\tau_{2}^{\frac34}
e^{\sqrt{\tau_{2}}r} 
&\mbox{if}\
r\ge(\Re\sqrt{\ii \tau})^{-1},
\end{array}\right.
\end{split}
\end{align*}
which leads to
\begin{align}\label{app.est1.coeff}
\begin{split}
|r^{-\frac32} I_{l+\frac12}(\sqrt{\ii \tau} r)|
\lesssim_{\,\tau_{2},l}
e^{\sqrt{\tau_{2}}r},
\quad r\in(0,\infty).
\end{split}
\end{align}
Next we estimate $\Lambda_{lm}(r,\tau)$ in \eqref{rep.cr} defined by \eqref{def.Lambda}. From
\begin{align*}
&
\big|
r^{-\frac32} K_{l+\frac12}(\sqrt{\ii \tau} r)
s^{l+\frac32} I_{l+\frac12}(\sqrt{\ii \tau} s)
\big|\\
&\lesssim_{\,l}
\left\{
\begin{array}{ll}
r^{-l-2}
s^{2l+2}
&\mbox{if}\
s \le r \le (\Re\sqrt{\ii \tau})^{-1},\\[5pt]
|\tau|^{\frac{l}{2}}
r^{-2}
e^{-\Re(\sqrt{\ii \tau})r}
s^{2l+2}
&\mbox{if}\ 
s \le (\Re\sqrt{\ii \tau})^{-1} \le r,\\[5pt]
|\tau|^{-\frac12}
r^{-2}
s^{l+1}
e^{-(\Re\sqrt{\ii \tau})(r-s)}
&\mbox{if}\ 
(\Re\sqrt{\ii \tau})^{-1} \le s \le r,
\end{array}\right.
\end{align*}
we see that
\begin{align*}
&
\bigg|
\int_{0}^{r}
r^{-\frac32} K_{l+\frac12}(\sqrt{\ii \tau} r)
s^{l+\frac32} I_{l+\frac12}(\sqrt{\ii \tau} s)
\dd s
\bigg|\\
&\lesssim_{\,l}
\left\{
\begin{array}{ll}
r^{l+1}
&\mbox{if}\
r \le (\Re\sqrt{\ii \tau})^{-1},\\[5pt]
|\tau|^{-\frac{l}{2}-\frac32}
r^{-2}
e^{-\Re(\sqrt{\ii \tau})r}
+ |\tau|^{-1}
r^{l-1}
&\mbox{if}\ 
r \ge (\Re\sqrt{\ii \tau})^{-1}
\end{array}\right.\\
&\lesssim_{\,l}
r^{l+1},
\quad r\in(0,\infty).
\end{align*}
In the similar manner, from
\begin{align*}
&
\big|
r^{-\frac32} I_{l+\frac12}(\sqrt{\ii \tau} r)
s^{l+\frac32} K_{l+\frac12}(\sqrt{\ii \tau} s)
\big|\\
&\lesssim_{\,l}
\left\{
\begin{array}{ll}
r^{l-1}
s
&\mbox{if}\
r \le s \le (\Re\sqrt{\ii \tau})^{-1},\\[5pt]
|\tau|^{\frac{l}2}
r^{l-1}
s^{l+1}
e^{-\Re(\sqrt{\ii \tau})s}
&\mbox{if}\ 
r \le (\Re\sqrt{\ii \tau})^{-1} \le s,\\[5pt]
|\tau|^{-\frac12}
r^{-2}
s^{l+1}
e^{(\Re\sqrt{\ii \tau})(r-s)}
&\mbox{if}\ 
(\Re\sqrt{\ii \tau})^{-1} \le r \le s,
\end{array}\right. 
\end{align*}
we see that
\begin{align*}
&
\bigg|
\int_{r}^{\infty}
r^{-\frac32} I_{l+\frac12}(\sqrt{\ii \tau} r)
s^{l+\frac32} K_{l+\frac12}(\sqrt{\ii \tau} s)
\dd s
\bigg|\\
&\lesssim_{\,l}
\left\{
\begin{array}{ll}
|\tau|^{-1}
r^{l-1}
&\mbox{if}\
r \le (\Re\sqrt{\ii \tau})^{-1},\\[5pt]
|\tau|^{-1}
r^{l-1}
&\mbox{if}\ 
r \ge (\Re\sqrt{\ii \tau})^{-1}
\end{array}\right.\\
&\lesssim_{\,l}
\tau_{1}^{-1}
r^{l-1},
\quad r\in(0,\infty).
\end{align*}
Then, by the definition \eqref{def.Lambda} and the estimate of $B_{lm}(\tau)$ in \eqref{est.B}, we have
\begin{align}\label{est.Lambda}
\begin{split}
&|\Lambda_{lm}(r,\tau)|
\lesssim_{\,\tau_{1},\tau_{2},l}
(1+r)^{l+1},
\quad r\in(0,\infty).
\end{split}
\end{align}
Consequently, from \eqref{app.est1.coeff}--\eqref{est.Lambda}, we obtain
\begin{align}\label{est.cr}
|c^{r}_{lm}(r)|
\lesssim_{\,\tau_{1},\tau_{2},l}
(|C^{r}_{lm}(\tau)| + 1)
e^{\sqrt{\tau_{2}}r},
\quad r\in(0,\infty).
\end{align}
Next we consider $c^{(1)}_{lm}(r,\tau)$ defined in \eqref{rep.c1}. Observe that
\begin{align*}
\frac{\dd c^{r}_{lm}}{\dd r}(r)
&=
C^{r}_{lm}
\frac{\dd}{\dd r}
\big(
r^{-\frac32}
I_{l+\frac12}(\sqrt{\ii \tau} r)
\big)\\
&\quad
-lB_{lm}
\int_{0}^{r} 
\frac{\dd}{\dd r}
\big(
r^{-\frac32} K_{l+\frac12}(\sqrt{\ii \tau} r)
\big)
s^{l+\frac32} I_{l+\frac12}(\sqrt{\ii \tau} s)
\dd s\\
&\quad
-
lB_{lm}
\int_{r}^{\infty}
\frac{\dd}{\dd r}
\big(
r^{-\frac32} I_{l+\frac12}(\sqrt{\ii \tau} r)
\big)
s^{l+\frac32} K_{l+\frac12}(\sqrt{\ii \tau} s) 
\dd s.
\end{align*}
Then, by the relation \eqref{eq.rel} and Lemma \ref{lem.MBF}, in the similar manner as above, we obtain
\begin{align}\label{est.c1}
|c^{(1)}_{lm}(r,\tau)|
&\lesssim_{\,\tau_{1},\tau_{2},l}
(|C^{r}_{lm}(\tau)| + 1)
e^{\sqrt{\tau_{2}}r}.
\end{align}
Finally, by Lemma \ref{lem.MBF}, we estimate $c^{(2)}_{lm}(r,\tau)$ defined in \eqref{rep.c2} as
\begin{align}\label{est.c2}
|c^{(2)}_{lm}(r,\tau)|
&\lesssim_{\,\tau_{2},l}
|C^{(2)}_{lm}(\tau)| e^{\sqrt{\tau_{2}}r}.
\end{align}
Then the assertion follows from \eqref{est.cr}--\eqref{est.c2}.
\end{proof}

Next we give the estimate of $C^{r}_{lm}$ in \eqref{def.Cr} and $C^{(2)}_{lm}$ in \eqref{def.C2}.

\begin{lemma}\label{lem.est.CrC2}
We have
\begin{align*}
|C^{r}_{lm}(\tau)|
+ |C^{(2)}_{lm}(\tau)|
&\lesssim_{\,v_{1},\tau_{1},\tau_{2},l,\rho}
1
\end{align*}
uniformly in $\tau\in\R$ with $0<\tau_{1}\le|\tau|\le\tau_{2}<\infty$.
\end{lemma}

\begin{proof}
By Lemma \ref{lem.MBF}, we see that $\mathcal{I}_{l+\frac12,j}(\sqrt{\ii \tau})$ in \eqref{def.Cr}--\eqref{def.C2} is estimated as
\begin{align}\label{est1.prf.lem.est.CrC2}
\mathcal{I}_{l+\frac12,j}(\sqrt{\ii \tau})
=
\int_{0}^{\rho} 
r^{j} |I_{l+\frac12}(\sqrt{\ii \tau} r)|^{2}
\dd r
\approx_{\,l,\rho}
|\tau|^{l+\frac12},
\quad
j=1,-1.
\end{align}
We then consider $C^{r}_{lm}$ defined in \eqref{def.Cr}. By the H\"{o}lder inequality, we have
\begin{align}\label{est2.prf.lem.est.CrC2}
|C^{r}_{lm}(\tau)|
\le
\frac{1}{\mathcal{I}_{l+\frac12,-1}(\sqrt{\ii \tau})^{\frac12}}
\bigg(
\int_{0}^{\rho}
|c^{r}_{lm}(r,\tau) - \Lambda_{lm}(r,\tau)|^{2}
r^{2} \dd r
\bigg)^{\frac12}.
\end{align}
The definition of $c^{r}_{lm}(r,\tau)$ in \eqref{exp.v1} implies that
\begin{align*}
\begin{split}
\bigg(
\int_{0}^{\rho}
|c^{r}_{lm}(r,\tau)|^{2}
r^{2} \dd r
\bigg)^{\frac12}
&=
\bigg(
\int_{0}^{\rho}
|\langle \hat{v}_{1}, \bm{Y}_{lm}\rangle_{S}(r,\tau)|^{2}
r^{2} \dd r
\bigg)^{\frac12}\\
&\le
\|\hat{v}_{1}(\tau)\|_{L^{2}(B_{\rho})},
\end{split}
\end{align*}
where Lemma \ref{lem.VSH.exp} (\ref{item1.lem.VSH.exp}) is used in the second line. The estimate \eqref{est.Lambda} shows that
\begin{align*}
\bigg(
\int_{0}^{\rho}
|\Lambda_{lm}(r,\tau)|^{2}
r^{2} \dd r
\bigg)^{\frac12}
\lesssim_{\,\tau_{1},\tau_{2},l,\rho}
1.
\end{align*}
Hence the triangle inequality applied to \eqref{est2.prf.lem.est.CrC2} and \eqref{est1.prf.lem.est.CrC2} give
\begin{align*}
|C^{r}_{lm}(\tau)|
\lesssim_{\,v_{1},\tau_{1},\tau_{2},l,\rho}
\mathcal{I}_{l+\frac12,-1}(\sqrt{\ii \tau})^{-\frac12}
\lesssim_{\,v_{1},\tau_{1},\tau_{2},l,\rho}
1.
\end{align*}
It is easy to check that $C^{(2)}_{lm}(\tau)$ satisfies the same estimate. This completes the proof.
\end{proof}

Now we can continue with the proof of Lemma \ref{lem.coeff.v1} and complete it.

\begin{proofx}{Lemma \ref{lem.coeff.v1} (second half)}
The estimate \eqref{est.coeff.v1} for the coefficients $c^{r}_{lm}, c^{(1)}_{lm}, c^{(2)}_{lm}$ is a consequence of the formulas \eqref{rep.cr}, \eqref{rep.c1} and \eqref{rep.c2} with \eqref{def.Cr}--\eqref{def.C2} combined with the estimates in Lemmas \ref{lem.est.crc1c2}--\ref{lem.est.CrC2}. This completes the proof of Lemma \ref{lem.coeff.v1}.
\end{proofx}

Thanks to Lemma \ref{lem.coeff.v1}, the truncation $v_{2}$ in Lemma \ref{lem.Trunc} can be defined in the whole space-time $\R^{4}$. Gathering the estimates \eqref{cons.lem.Ext}, \eqref{capp} and \eqref{ineq.lem.Trunc} and taking $\ep>0$ smaller again, we see that $v_{2}$ approximates $v$, up to $\ep$, in the topology of $C(K)^{3}$. Next we consider the pressure $q_{2}$ associated with $v_{2}$. Using the same $L,\tau_{1},\tau_{2}$ in Lemma \ref{lem.Trunc}, we define
\begin{align}\label{def.q2}
\begin{split}
q_{2}(x,t)
=q_{2}(r,\theta,\phi,t)
:=
\frac{1}{2\pi}
\sum_{l=1}^{L} \sum_{m=-l}^{l}
\int_{\tau_{1} < |\tau| < \tau_{2}}
b_{lm}(r,\tau) Y_{lm}
e^{\ii \tau t}
\dd \tau.
\end{split}
\end{align}
with $b_{lm}(r,\tau)$ in Lemma \ref{lem.coeff.q1}. Then the restriction $(v_{2},q_{2})|_{\R^{4}_{+}}$ satisfies \eqref{app} and the bound \eqref{growth.thma}. Therefore, putting $(u,p)=(v_{2},q_{2})|_{\R^{4}_{+}}$, we complete the proof of Theorem \ref{thma}.

   \subsection{Discussions on asymptotics at infinity}
   \label{Disc}

Here we give some details of Theorem \ref{thma} and discuss the comparison with \cite[Theorem 1.2]{EGFPS2019} for the heat equation where global approximations are decaying at infinity.
\begin{itemize}
\item
Since $v_{2}$ and $q_{2}$ are defined in \eqref{def.v2} and \eqref{def.q2}, respectively, the estimate \eqref{est.v1} and the Riemann-Lebesgue theorem imply that $(u,p)=(v_{2},q_{2})|_{\R^{4}_{+}}$ decays in time
\begin{align}
\lim_{t\to\infty} \big(u(x,t),p(x,t)\big)=0
\end{align}
for each fixed $x\in\R^{3}$. Thus $(u,p)$ on $\R^{4}_{+}$ grows in space but decays in time.

\item
In the case of the heat equation, the estimate $|v_{2}(x,t)|\le Ce^{c|x|}$ enables us to appeal to the uniqueness result \cite[Section 7.1 (b)]{Joh1991book}. Then we see that $v_{2}$ is uniquely represented by the fundamental solution of the heat operator with the given data $f(x):=v(x,0)$. Therefore, by cutting off $f(x)$, we conclude that $v_{2}$ can be approximated in the topology of $C(K)$ by $u=e^{t\Delta} u_{0}$ using the heat semigroup $e^{t\Delta}$ and a smooth, compactly supported initial data $u_{0}$. Hence the global approximation with decay at infinity follows for the heat equation. We emphasize that the procedure above is not applicable to the Stokes system as the uniqueness of solutions without decay at infinity is impossible as implied by parasitic solutions of the form \eqref{para}.
\end{itemize}

    \section{Proof of Theorem \ref{thmb}}
    \label{sec-proofB}

In this section, we prove Theorem \ref{thmb}. Without loss of generality, by translation and rescaling, we may consider that $K=[-2,2]^{3}\times[1,3]\subset\R^{4}_{+}$. Set $\tilde K = [-1,1]^{3}\subset\R^{3}$.

    \subsection{A local parabolic Harnack inequality}

We first recall the following local parabolic Harnack inequality on $\tilde K\times[2,3]$ for nonnegative solutions of the heat equations in $K$; see \cite[Section 7.1, Theorem 10]{Eva2010book}.

\begin{lemma}\label{HH}
There is a constant $C_{0} >0$ such that, for any nonnegative smooth solution $w$ of $\partial_{t} w - \Delta w = 0$ in $K$, we have
\[
\sup_{x\in \tilde K} w(x,2)
\le C_{0} \inf_{x\in \tilde K} w(x,3).
\]
\end{lemma}

In the sequel we assume, without loss of generality, that $C_{0} \geq 1$.

    \subsection{Proof by contradiction}

Set $\ep_{0}=(2C_{0})^{-1}$ and take a decreasing function $c \in C^\infty (\R)$ such that $c(2)=2$ and $c(3)=\ep_{0}$. Then, as the smooth parasitic solution $(v,q)$ of the form \eqref{para}, we define
\[
v(x,t)
=
\left(
\begin{matrix}
c(t)&0&0
\end{matrix}
\right)^{\top}
\qquad \text{and} \qquad
q(x,t) = -c'(t)x^{1}.
\]
Moreover, we define $\ep = \ep_{0}/2$. Notice that $\ep<\ep_{0}\le 1/2$.

Let us reason by contradiction. Assume that $(u,p)$ in the statement satisfies
\begin{equation}\label{siapp}
\|v - u\|_{C(K)} \le \ep.
\end{equation}
Operating $\nabla\cdot$ to the first line of \eqref{app}, we see that $\Delta p=0$ in $\R^{4}_{+}$, namely, that $p=p(x,t)$ is harmonic in $x$. Then, by the Liouville theorem, $p$ is, for each time, constant in $x$. Using again \eqref{app}, this implies that $u$ actually satisfies the heat equation $\partial_{t} u - \Delta u = 0$ in $\R^{4}_{+}$. Moreover, from \eqref{siapp} and from the fact that $c$ is greater than $\ep_{0}$ on $[1,3]$, we deduce that the first component $u^{1}$ of $u$ is nonnegative in $K$. It is therefore a nonnegative smooth solution of the heat equation in $K$ so that we can apply Lemma \ref{HH} and get 
\begin{equation}\label{C1}
\sup_{x\in \tilde K} u^{1}(x,2)
\le C_{0} \inf_{x\in \tilde K} u^{1}(x,3).
\end{equation}
But \eqref{siapp} implies that
\begin{equation}\label{C2}
\sup_{x\in \tilde K} u^{1}(x,2) \ge c(2) - \ep > \frac32
\qquad \text{and} \qquad
\inf_{x\in \tilde K} u^{1}(x,3) \le c(3) + \ep < \frac{1}{C_{0}}.
\end{equation}
Hence, by combining \eqref{C1} and \eqref{C2}, we get $3/2<1$. This is the contradiction which completes the proof of Theorem \ref{thmb}.

    \section*{Acknowledgment}

F.S. was supported by the project ANR-23-CE40-0014-01 BOURGEONS of the French National Research Agency (ANR).

    \addcontentsline{toc}{section}{References}
    \bibliography{Ref}
    \bibliographystyle{alpha}

\begin{flushleft}
M.~Higaki, 
\textsc{Department of Mathematics, Graduate School of Science, Kobe University, 1-1 Rokkodai, Nada-ku, Kobe 657-8501, Japan}\par\nopagebreak
\noindent  \textit{E-mail address: }\texttt{higaki@math.kobe-u.ac.jp}
\end{flushleft}

\begin{flushleft}
F.~Sueur, \textsc{Department of Mathematics, 
Maison du nombre, 6 avenue de la Fonte, 
University of Luxembourg,  
L-4364 Esch-sur-Alzette, Luxembourg }\par\nopagebreak 
\noindent  \textit{E-mail address:} \texttt{Franck.Sueur@uni.lu}
\end{flushleft}

\medskip

    \noindent\today

\end{document}